\documentclass[a4paper, 12pt]{amsart}

\usepackage{cite}

\pdfoutput=1
\usepackage{latexsym, amsmath, amssymb, amsthm, mathrsfs, bm, mathtools,comment}
\usepackage[mathscr]{eucal}
\usepackage{subfigure}

\usepackage{array}
\usepackage{tikz}
\usetikzlibrary{cd}
\usepackage{aliascnt}
\usepackage{graphicx}
\usepackage[T1]{fontenc}
\usepackage{mathptmx}
\usepackage{microtype}
\usepackage[colorlinks=true,linkcolor=blue,urlcolor=blue]{hyperref}

\usepackage[onehalfspacing]{setspace}

\usepackage[inner=2.4cm,outer=2.4cm, bottom=3.2cm]{geometry}

\DeclareMathAlphabet{\altmathcal}{OMS}{cmsy}{m}{n}

\allowdisplaybreaks

\usepackage{xcolor, color, soul}

\usepackage{color}
\usepackage[all]{xy}
\usepackage{enumerate}
\usepackage{mathbbol,mathtools}

\renewcommand{\geq}{\geqslant}

\newtheorem{theorem}{Theorem}[section]

\newtheorem{headthm}{Theorem}

\newaliascnt{headcor}{headthm}

\aliascntresetthe{headcor}

\newaliascnt{headconj}{headthm}

\aliascntresetthe{headconj}

\newaliascnt{corollary}{theorem}
\newtheorem{corollary}[corollary]{Corollary}
\aliascntresetthe{corollary}

\newaliascnt{claim}{theorem}

\aliascntresetthe{claim}

\newaliascnt{lemma}{theorem}
\newtheorem{lemma}[lemma]{Lemma}
\aliascntresetthe{lemma}

\newaliascnt{conjecture}{theorem}
\newtheorem{conjecture}[conjecture]{Conjecture}
\aliascntresetthe{conjecture}

\newaliascnt{proposition}{theorem}
\newtheorem{proposition}[proposition]{Proposition}
\aliascntresetthe{proposition}

\theoremstyle{definition}
\newaliascnt{definition}{theorem}
\newtheorem{definition}[definition]{Definition}
\aliascntresetthe{definition}

\newaliascnt{notation}{theorem}

\aliascntresetthe{notation}

\newaliascnt{example}{theorem}
\newtheorem{example}[example]{Example}
\aliascntresetthe{example}

\newaliascnt{examples}{theorem}

\aliascntresetthe{examples}

\newaliascnt{remark}{theorem}
\newtheorem{remark}[remark]{Remark}
\aliascntresetthe{remark}

\newaliascnt{fact}{theorem}

\aliascntresetthe{fact}

\newaliascnt{question}{theorem}
\newtheorem{question}[question]{Question}
\aliascntresetthe{question}

\newaliascnt{questions}{theorem}

\aliascntresetthe{questions}

\newaliascnt{problem}{theorem}

\aliascntresetthe{problem}

\newaliascnt{construction}{theorem}

\aliascntresetthe{construction}

\newaliascnt{setup}{theorem}
\newtheorem{setup}[setup]{Setup}
\aliascntresetthe{setup}

\newaliascnt{algorithm}{theorem}

\aliascntresetthe{algorithm}

\newaliascnt{observation}{theorem}

\aliascntresetthe{observation}

\newaliascnt{discussion}{theorem}

\aliascntresetthe{discussion}

\newaliascnt{defprop}{theorem}

\aliascntresetthe{defprop}

\def\equationautorefname~#1\null{(#1)\null}
\def\sectionautorefname~#1\null{Section #1\null}
\def\subsectionautorefname~#1\null{\S #1\null}

\def \Spec{{\operatorname{Spec\, }}}

\def\Ass{\operatorname{Ass}}

\def \conv{{\operatorname{conv}}}

\def \lcm{{\operatorname{lcm}}}

\def \ord{{\operatorname{ord}}}
\def \NP{{\operatorname{NP}}}
\def \RV{{\operatorname{\mathscr{RV}}}}

\DeclareMathOperator{\spann}{span}
\DeclareMathOperator{\gr}{gr}

\DeclareMathOperator{\Quot}{Quot}

\DeclareMathOperator{\GL}{GL}

\def \f1{\mathbf{1}}
\def\ba{\mathbf{a}}

\newcommand{\kk}{\mathbb{k}}

\def\xi{x}

\def\ls{\leqslant}
\def\gs{\geqslant}

\def\uv{\underline{v}}

\def\fm{\mathfrak{m}}
\def\frp{\mathfrak{p}}
\def\frq{\mathfrak{q}}

\def\fX{\mathbf{X}}

\def\fb{\mathfrak{b}}
\def\fc{\mathfrak{c}}

\def\bn{\mathbf{n}}

\def\bh{\mathbf{h}}

\def\bmm{\mathbf{m}}

\def \CC{\mathbb C}
\def \QQ{\mathbb Q}
\def \RR{\mathbb R}
\def \KK{\mathbb K}

\def \NN{ {\mathbb Z}_{\gs 0}}
\def \ZZz{ {\mathbb Z}_{> 0}}

\def \QQez{ {\mathbb Q}_{\gs 0}}
\def \ZZ{\mathbb Z}

\def \C{\mathcal C}
\def \T{\mathcal T}
\def \cS{\mathcal S}

\def \O{\mathcal O}
\def \G{\mathcal G}
\def \R{\mathcal R}
\def \V{\mathcal V}
\def \H{\mathcal H}
\def \J{\mathcal J}
\def \I{\mathcal I}
\def \P{\mathcal P}

\def \B{\mathcal B}

\def \O{\mathcal O}
\def \F{\mathcal F}
\def \W{\mathcal W}

\def \cP{\mathscr P}
\def \cS{\mathscr S}
\def \cT{\mathscr T}

    \def\fx{\mathbf{x}}

\def\fe{\mathbf{e}}

\dedicatory{{Dedicated to Bernd Ulrich on the occasion of his seventieth birthday.}}

\begin{document}

\title[Rational powers, invariant ideals, and the summation formula]{Rational powers, invariant ideals, and the summation formula}

\author[Sankhaneel Bisui]{Sankhaneel Bisui}
\address{Sankhaneel Bisui\\School of Mathematical and Statistical Sciences, Arizona State University, P.O. Box 871804, Tempe, AZ 85287-18041, USA. \emph{Email:} {\rm sankhaneel.bisui@asu.edu}}

\author[Sudipta Das]{Sudipta Das}
\address{Sudipta Das\\School of Mathematical and Statistical Sciences, Arizona State University, P.O. Box 871804, Tempe, AZ 85287-18041, USA. \emph{Email:} {\rm sdas137@asu.edu}}

\author[T\`ai Huy H\`a]{T\`ai Huy H\`a$^1$}
\address{T\`ai Huy H\`a\\Mathematics Department, Tulane University, 6823 St. Charles Avenue, New Orleans, LA 70118, USA. \emph{Email:} {\rm tha@tulane.edu}}
 \thanks{$^{1}$ The third author  was partially supported by the Simons Foundation.}

\author[Jonathan Monta{\~n}o]{Jonathan Monta{\~n}o$^2$}
\address{Jonathan Monta{\~n}o\\School of Mathematical and Statistical Sciences, Arizona State University, P.O. Box 871804, Tempe, AZ 85287-18041, USA. \emph{Email:} {\rm montano@asu.edu}}
 \thanks{$^{2}$ The fourth author was partially funded by NSF Grant DMS \#2001645/2303605 and DMS \#2401522.}

  \begin{abstract}
We provide  explicit descriptions  for the rational powers  and Rees valuations of several classes of ideals 
invariant under  natural actions of tori and products of general linear groups, in terms of 
polyhedra and lattice points. This allows us to show that a
version of Musta\c{t}\u{a}-Takagi's  summation formula  for multiplier ideals also holds for the  rational powers of these ideals.
 Moreover, for arbitrary ideals in normal domains that are finitely generated over algebraically closed fields, we prove a weaker version of this formula that holds for
sufficiently large rational numbers.  
  \end{abstract}


\maketitle


\section{Introduction}

Rees valuations and rational powers of ideals play a central role in many results in algebraic geometry and commutative algebra \cite{hong2014specialization, Itoh88,DeStefani16, Ciuperca21, Ciuperca20, Lewis23, computing_rat_pow_23, knutson2005balanced, ValeryKnutson10, DiPasq_et_al_19, RVS25}.  These two sets of invariants are so closely related that each uniquely determines the other (see \autoref{rem:determine}). However, explicitly computing them is  a difficult task because it requires determining the order valuations for every prime divisor arising in the exceptional divisor of the normalized blowup of the ideal \cite{Rees56}. In fact, complete descriptions of Rees valuations in the literature  are limited to  monomial ideals, where they correspond to supporting hyperplanes of  Newton polyhedra \cite[Theorem 10.3.5]{huneke2006integral}. One of the primary goals of this paper is to provide explicit descriptions of both Rees valuations and rational powers for several classes of ideals in polynomial rings that are invariant under  natural actions of tori and products of general linear groups.

Originally considered by Samuel and Rees in the 1950s \cite{Samuel52, Rees56}, the notion of rational powers of an ideal 
$I$ in a Noetherian domain  
$R$ was introduced  by Lejeune-Jalabert and Teissier \cite[\S 4.2]{lejeune2008cloture}. The power $\overline{I^w}$ for $w\in \QQ_{\gs 0}$ is defined as $\overline{I^w}:= \cap_v \{x \mid v(x) \geq w v(I)\}$, where the intersection is taken over all the discrete valuations of $\Quot(R)$ that are nonnegative in $R$, or alternatively, over the Rees valuations of $I$.  This construction provides a powerful tool for understanding the structure of filtrations of ideals such as integral closures \cite{hong2014specialization,Itoh88}, symbolic powers \cite{DiPasq_et_al_19}, Frobenius powers \cite{RVS25}, and multiplier ideals \cite[\S9.6]{LAZARSFELD}.  Despite its historical roots,  rational powers have  seen limited development, as noted by Knutson in 2010, who observed that this {\it ``beautiful construction has sat essentially unused in the literature for over 50 years now"} \cite{Knutson_MathOverflow_Rational}.

Let $X = X_{m \times n}$ be a generic  $m \times n$ matrix, meaning its entries are distinct variables. The polynomial ring $\mathbb{C}[X]$ carries a natural action of $\mathrm{GL}_m(\mathbb{C}) \times \mathrm{GL}_n(\mathbb{C})$ defined by $
(A, B) \cdot X = A^{-1} X B
$,
where \(A \in \mathrm{GL}_m(\mathbb{C})\) and \(B \in \mathrm{GL}_n(\mathbb{C})\). The ideals in  the polynomial ring $\CC[X]$ that are invariant under this action have been extensively studied in the literature, see e.g. \cite{deConcini80,deConcini76,Bruns91,BRUNS_VETTER,Henriques_Varbaro,JEFFRIES_MONTANO_VARBARO,DeSMNB24,BrC98}. 
The combination of representation theory and the theory of algebras with straightening laws has led to a wide range of explicit results about  the structure and properties of these  ideals. 
If one considers a symmetric generic matrix $Y=Y_{n\times n}$ and $\GL_n(\CC)$-action on $\CC[Y]$ given by $A\cdot Y=AYA^t$, or a skew-symmetric generic matrix $Z=Z_{n\times n}$ with  $\GL_n(\CC)$-action on $\CC[Z]$ given by $A\cdot Z=AZA^t$, one obtains analogous well-studied invariant ideals, see e.g. the references above and \cite{Abeasis80,Abeasis80Young,DeNegri96,Conca94}.

We introduce the concept of {\it Rees package}, which is a combinatorial characteristic-free construction that encapsulates the description of both Rees valuations and rational powers in terms of polyhedral data (see \autoref{def_Rees_Package}). In our first theorem, we show that the  classes of invariant ideals described above admit such a Rees package. Furthermore, in arbitrary characteristic,  we show that  monomial ideals in affine semigroup rings,  sums of products of determinantal ideals of generic, symmetric, and skew-symmetric matrices, and products of determinantal ideals of Hankel matrices, also admit  a Rees package. 

\begin{headthm}
	\label{thmB}
	Let $\kk$ be an arbitrary field. The following classes of ideals admit a Rees package. In particular, their 	
	rational powers and Rees valuations can be explicitly  described in terms of	lattice points  and supporting hyperplanes of  polyhedra.
	\begin{enumerate}[{\rm (a)}]
		\item{\rm (}\autoref{mon.aff.rees}{\rm )}  Monomial ideals in affine semigroup rings.
		\item{\rm (}\autoref{gen_det_rees}, \autoref{cor_inv_gen}{\rm )}  Sums of products of determinantal ideals in $\kk[X_{m\times n}]$ of a generic matrix $X_{m\times n}$. If $\kk=\CC$, more generally  any $\GL_m(\CC) \times \GL_n(\CC)$-invariant ideal of $\CC[X_{m\times n}]$.
		\item{\rm (}\autoref{gen_sym_rees}, \autoref{cor_inv_sym}{\rm )}  Sums of products of determinantal ideals in $\kk[Y_{n\times n}]$ of a  generic symmetric matrix $Y_{n\times n}$. If $\kk=\CC$, more generally  any $\GL_n(\CC)$-invariant ideal of  $\kk[Y_{n\times n}]$
		\item{\rm (}\autoref{gen_sym_pff}, \autoref{cor_inv_Pf}{\rm )}  Sums of products of  ideals of Pfaffians in $\kk[Z_{n\times n}]$ of a skew-symmetric generic matrix $Z_{n\times n}$.  If $\kk=\CC$, more generally  any $\GL_n(\CC)$-invariant ideal of $\CC[Z_{n\times n}]$.
		\item{\rm (}\autoref{gen_Hank}{\rm )}  Products of determinantal ideals  in $\kk[W]$ of a Hankel matrix   of variables $W$.
	\end{enumerate}
\end{headthm}



In \cite{mustacta2002multiplier}, Musta\c{t}\u{a} showed a summation formula for multiplier ideals of pairs ideal sheaves, and this  formula 
 was later generalized by Takagi  \cite{takagi2006formulas} and by Jow-Miller \cite{Jow_Miller08}. Musta\c{t}\u{a}-Takagi  summation formula  raises the natural question of whether similar identities hold for symbolic powers and rational  powers of  ideals. The former was 
established in \cite{ha2020symbolic, HJKN2023} and the latter in  \cite{banerjee2023integral} for monomial ideals. Thus we propose the following conjecture (cf. \cite[Question 2.4]{banerjee2023integral}).

\begin{conjecture}\label{The_Question}
Let $\kk$ be an algebraically closed field and
$R$ and $S$   normal domains finitely generated over  $\kk$.  Set $T:=R\otimes_\kk S$. 
The following  formula holds for every $w\in \QQ_{\gs 0}$ and nonzero ideals $I\subset R$ and
 $J\subset S$
\begin{equation*}\label{eqn_The_Q}
\overline{\left(IT+JT\right)^w}= \sum_{\stackrel{0\ls \alpha\ls w,}{{\scriptscriptstyle \alpha\in \QQ}}}\overline{I^\alpha}\,\overline{J^{w-\alpha}}\,T.
\end{equation*}
\end{conjecture}
We devote the  last two sections of the paper to tackle this conjecture. 
First, we note that the restriction on the field is partially justified  by  a related question (see  \autoref{question_rees_product} and \autoref{example.fields.algebraically.closed}). We also note that the use of two different $\kk$-algebras is essential in this conjecture, as the summation formula for rational powers does not hold in general if instead one considers two ideals in the same ring, see  \autoref{same.ring.no}. 


In \autoref{prev_main_package} we show that if two ideals $I\subset R$ and $J\subset S$ admit a Rees package, e.g. the invariant ideals in \autoref{thmB}, then $IT+JT$ does as well. Moreover, a Rees package for $IT+JT$ can be obtained by `gluing' those of $I$ and $J$. In particular, this shows that 
 the Rees valuations of $IT + JT$ are obtained  by pairing the ones of $I$ and $J$ (see \autoref{cor_Rees_val}). 
This result allows us to compute rational powers of more classes of ideals, see for instance \autoref{example_cool}. 

 In our next theorem, we use \autoref{prev_main_package}  to settle \autoref{The_Question} for ideals with a Rees package without restrictions on the field. 

\begin{headthm}[\autoref{main_Rees_P}, \autoref{cor_invs_all}]
	\label{thmC}
	Adopt the notation in \autoref{The_Question} with $\kk$ an arbitrary field.  The conclusion of \autoref{The_Question} holds if $I$ and $J$ admit a Rees package. In particular, this holds if $I$ and $J$ belong to any of the  classes of ideals in \autoref{thmB}, and not necessarily both from the same class. 
\end{headthm}

In the following 
theorem, we provide further evidence for \autoref{The_Question}. Here we prove an inclusion for arbitrary ideals and $w\gg 0$. This result  can be seen as a weaker version of \autoref{The_Question} since we always have
$$
\sum_{\stackrel{0\ls \tau\ls w,}{{\scriptscriptstyle \tau\in \QQ}}}\overline{I^\tau}\,\overline{J^{w-\tau}}\,T
= \bigcap_{\stackrel{0\ls \tau\ls w,}{{\scriptscriptstyle \tau\in \QQ}}}\left(	\overline{I^\tau}T + \overline{J^{w-\tau}}T\right),$$
see \autoref{remark_weaker}.  

\begin{headthm}[\autoref{thm_asymp_main}]
	\label{thmA}
	Following the notation in \autoref{The_Question}, 
there exists  $w_0\in \QQ_{>0}$ such that if $w_0\ls \tau\ls w-w_0$, $\tau\in \QQ_{\geqslant 0}$ then 
	 $$\sum_{\stackrel{0\ls \alpha\ls w,}{{\scriptscriptstyle \alpha\in \QQ}}}\overline{I^\alpha}\,\overline{J^{w-\alpha}}\,T
	 \subseteq
	 \overline{\left(IT+JT \right)^w}
	 \subseteq
	 \overline{I^\tau}T+\overline{J^{w-\tau}}T.
	 $$
\end{headthm}

The proof of \autoref{thmA} is  influenced  by that of \cite[Theorem 0.3]{mustacta2002multiplier}.  The first step is to define rational powers of ideal sheaves in  normal  varieties inspired by \cite[D\'efinition 4.2.2]{lejeune2008cloture} and \cite[Definition 9.6.2]{LAZARSFELD}, see \autoref{def_int_clo}. 
Next, we prove the statement for  ideals in discrete valuation rings, see \autoref{dvr}. 
Next, a careful use of  Serre's vanishing theorem gives us a  desired  vanishing result, see \autoref{vanishing}.  
Finally, we put it all together in \autopageref{proof_thm_asymp}.

We now outline the content of the paper. \autoref{section_rat_powers} includes general details about rational powers and Rees valuations,  as well as  the definition of rational powers of ideals sheaves.     \autoref{sec_Rees_Rat} contains the definition of Rees package and the proof of \autoref{thmB}.  \autoref{sec_package_summ} includes  the proof of \autoref{thmC}. This section also includes examples where our results are applied.
\autoref{sec_asymp_form} is devoted to the  proof of \autoref{thmA}.

\medskip

\begin{center}
 {\it Acknowledgments}
\end{center}

The authors are grateful to Robert Lazarsfeld for several 
enlightening insights about multiplier 
ideals, integral closures, and vanishing theorems.
We also thank Claudiu Raicu for helpful correspondence related to the proof  of \autoref{gen_det_rees} and Bernard Teissier for pointing out \autoref{rem_NND}. The fourth author thanks Alessandro De Stefani and Alexandra Seceleanu for discussions related to \autoref{example_subs_arr}. The third author  was partially supported by the Simons Foundation. The fourth author was partially funded by NSF Grant DMS \#2001645/2303605 and DMS \#2401522. We acknowledge support by the NSF Grant  DMS \#1928930, while the second and fourth authors were in residence at the Simons Laufer Mathematical Science Institute (formerly MSRI) in Berkeley, California, during the Spring 2024 semester.

\section{Rational powers and Rees valuations}\label{section_rat_powers}
In this section 
 we review and extend some relevant facts about 
 rational powers and Rees valuations.  
 Our main references in this subject is \cite{huneke2006integral}. 
We note that several of the statements in this section have been previously established  for the integral closures of powers of ideals. 
Our goal is to show that these properties continue to hold in the more general setting of rational powers and to provide    references for the corresponding assertions. 
We restrict our exposition to the following setup where all of our results occur.

\begin{setup}\label{setup_Back_rat}
Let $(R, \fm, \kk)$ be a Noetherian local ring, or positively  graded over $\kk$ with $\fm=R_{>0}:=\oplus_{n\in \ZZz} R_n$. We also assume $R$ is an integral domain and denote by $\Quot(R)$ its field of fractions. We set $I\subset R$ to be an ideal.
\end{setup}

\begin{definition}[Rational powers]\label{def_int_clo_rati}
Adopt \autoref{setup_Back_rat}.	
An element $x\in R$ is said to be {\it integral over $I$} if there exists $n\in \NN$ and  $a_i \in I^i$ such that $x^n+a_1x^{n-1}+a_2x^{n-2}+\cdots+a_n=0$.
	The set of all such element is an ideal denoted by $\overline{I}$ and called the {\it integral closure of $I$}.  Let $w=\frac{p}{q}$ be a nonnegative rational number with $p\in \NN$ and $q\in \ZZz$. The {\it $w$-th rational power of $I$} is the ideal defined by
	$$\overline{I^w}:=\{x\in R\mid x^q\in \overline{I^p}\}.$$
\end{definition}

\begin{remark}
It follows from the definitions that $\overline{I^w}$ is well-defined, i.e., if $\frac{p}{q}=\frac{p'}{q'}$ are two representations of $w$, then $\{x\in R\mid r^q\in \overline{I^p}\}=\{x\in R\mid r^{q'}\in \overline{I^{p'}}\}$. Furthermore, $\overline{I^0}=R$ and if $n\in \NN$ then $\overline{I^{\frac{n}{1}}}=\overline{I^n}$.
\end{remark}

In the following proposition we show that rational powers localize.

\begin{proposition}\label{localize}
Adopt \autoref{setup_Back_rat}. For every multiplicative subset $W\subset R$ and rational number $w\in \QQez$ we have $W^{-1}\overline{I^w}=\overline{(W^{-1}I)^w}$.
\end{proposition}
\begin{proof}
If $w$ is an integer, the statement follows from \cite[Proposition 10.5.5]{huneke2006integral}. Now let $w$ be arbitrary and write $w=\frac{p}{q}$. If $x\in \overline{I^w}$, then $(W^{-1}x)^q=W^{-1}x^q\in W^{-1}\overline{I^p}=\overline{(W^{-1}I)^p}$ and so $W^{-1}x\in \overline{(W^{-1}I)^w}$. On the other hand, if $y\in \overline{(W^{-1}I)^w}$, then $y^q\in\overline{(W^{-1}I)^p}=W^{-1}\overline{I^p}$. Therefore there exists $a\in W$ such that $ay^q\in \overline{I^p}$. The latter  implies $(ay)^q\in \overline{I^p}$ and so $ay\in\overline{I^w}$. Thus $y\in W^{-1}\overline{I^w}$, finishing the proof.
\end{proof}

A valuation is {\it discrete} if it is integer valued. Integral closures and rational powers can also be defined via  discrete valuations: by \cite[Theorem 6.8.3(3)]{huneke2006integral} we have that $x\in \overline{I}$ if and only if $x\in IV$ for every discrete valuation ring  $R\subseteq V\subseteq \Quot(R)$. Moreover, the latter condition can be verified by using finitely many valuations,  the Rees valuations of $I$.

\begin{definition}[Rees valuations]\label{def.Rees.val}
Adopt \autoref{setup_Back_rat}. Let $I=(a_1, \dots, a_d)$ be a set of generators and,  for each $i=1, \dots, d$, set $R_i=R\left[ \frac{a_1}{a_i},\ldots, \frac{a_d}{a_i}\right]$ and $\overline{R_i}$  its   integral closure in $\Quot(R_i)=\Quot(R)$. Let $\RV(I)$ be the set of distinct  discrete valuation rings arising as $V=\left( \overline{R_i}\right)_\frp$, where $\frp$ varies over the minimal primes of $I\overline{R_i}=a_i\overline{R_i}$. The set $\RV(I)$ is called the {\it Rees valuations rings of $I$}. We will often abuse notation and identify each  $V\in \RV(I)$ with any valuation $v$ which is a positive integer multiple  of the order valuation of $V$.
\end{definition}

We say that a discrete valuation $v$ of $\Quot(R)$  is {\it centered in $R$}, or simply that $v$ is a {\it valuation of $R$},  if $v(x)\gs 0$ for every $x\in R$. If $v$ is a valuation of $R$, then set $v(I)=\min\{v(x)\mid x\in I\}$.

\begin{remark}\label{discrete_local}
Let  $\RV(I)=\{v_1, \dots, v_r\}$ and  set $e:=\lcm(v_1(I), \dots, v_r(I))$ (or any multiple of it). For every $w\in \QQez$ and  $n:= \lceil we\rceil$ we have  $\overline{I^w}=\overline{I^{\frac{n}{e}}}$ (see \cite[Proposition 10.5.5]{huneke2006integral}). Moreover, $\overline{I^{\frac{1}{e}}}=\sqrt{I}$ \cite[Example 10.5.4]{huneke2006integral}.
\end{remark}

As mentioned above, Rees valuations can be used to define integral closures.

	\begin{theorem}[{\cite[Theorem 10.2.2]{huneke2006integral}}]\label{thm_Rees_val}
Adopt \autoref{setup_Back_rat} and the notation from \autoref{def.Rees.val}. For every $n\in \NN$ we have
$$\overline{I^n}=\cap_{i=1}^d \left(I^n \overline{R_i} \cap R\right)= \cap_{i=1}^d \left(a_i^n\overline{R_i}\cap R\right)
=
\cap_{V\in \RV(I)} \left(I^nV\cap R\right)= \{x\in R\mid v(x)\gs nv(I), \forall v\in \RV(I)\},$$
and no $V$ or $v$ can be  omitted without breaking some of  these equalities.
	\end{theorem}

The following lemma shows that  rational powers can be defined via the same valuations defining integral closures.

\begin{proposition}\label{lemma_same_vals}
Adopt \autoref{setup_Back_rat} and the notation from \autoref{def.Rees.val}. Let $\mathscr{T}$ be a set of   discrete   valuations of $R$. Then  $\overline{I^n} =\{x\in R\mid v(x)\gs nv(I), \forall v\in \mathscr{T}\} $ for every $n \in \NN$ if and only if  $\overline{I^w} =\{x\in R\mid v(x)\gs wv(I), \forall v\in \mathscr{T}\} $
for every $w \in \QQez$.
In particular, for every $w\in \QQez$, we have
$$\overline{I^w}=
\{x\in R\mid v(x)\gs wv(I), \forall v\in \RV(I)\},$$
and no 
$v$ can be omitted without breaking some of  these equalities.
\end{proposition}
\begin{proof}
The backwards implication is trivial. For the forward implication, fix $w\in \QQez$ and write $w=\frac{p}{q}\in \QQez$ with $p\in \NN$ and $q\in \ZZz$. Since $x\in \overline{I^w}$  if and only if $x^q\in \overline{I^p}$, and this is the case if and only if $qv(x)\gs pv(I)$ for every $v\in \mathscr{T}$, the first statement follows. The second statement follows from the first one and \autoref{thm_Rees_val}.
\end{proof}

The previous two results lead to the following useful fact about  associated primes of rational powers.

\begin{proposition}\label{assoc_prime}
Adopt \autoref{setup_Back_rat}. For each $V\in \RV(I)$, set $\frp_V:=\{x\in R\mid v(x)>0\}$.  We have
$$\cup_{w\in \QQez}\Ass(R/\overline{I^w})=\cup_{n\in \NN}\Ass(R/\overline{I^n})=\{\frp_V\mid V\in \RV(I)\}.$$
\end{proposition}
\begin{proof}
The second equality is proved in \cite[Discussion 10.1.3]{huneke2006integral} and the inclusion $\supseteq$ in the first one is trivial. It remains to show the first inclusion $\subseteq$.  For each $V\in \RV(I)$, let $\eta_V$ be its maximal ideal. From  \autoref{lemma_same_vals} we obtain
$\overline{I^w}=\cap_{V\in \RV(I)}(\eta_V^{\lceil wv(I)\rceil}\cap R)$. The result follows as each  ideal $\eta_V^{\lceil wv(I)\rceil}\cap R$ is $\frp_V$-primary.
\end{proof}

The following is a useful characterization of Rees valuations. 

\begin{proposition}\label{prop_min_vals}
Adopt \autoref{setup_Back_rat}. Let $I\subseteq R$ and  $v_1, \dots, v_r$  be discrete   valuations of  $R$. 
Then the following statements are equivalent.
\begin{enumerate}[\rm (i)]
\item $\RV(I)=\{v_1\ldots, v_r\}$.
\item $\overline{I^n}=\{x \in R \mid v_i(x)\geqslant n v_i(I), \forall i = 1,\ldots, r  \}$ for all  $n \in \NN $, and no $v_i$ can be omitted without breaking some of  these equalities.
\item  $\overline{I^w}=\{x \in R \mid v_i(x)\geqslant w v_i(I), \forall i = 1,\ldots, r  \}$ for all  $w \in \QQez $, and no $v_i$ can be omitted without breaking some of  these equalities.
\end{enumerate}
\begin{proof}
The implication (i) $\Rightarrow$ (ii) follows from \autoref{thm_Rees_val}, and (ii) $\Rightarrow$ (iii) from \autoref{lemma_same_vals}. It remains to show (iii) $\Rightarrow$ (i), but this is  content of \cite[Theorem 10.1.6]{huneke2006integral}, finishing the proof.
\end{proof}
\end{proposition}

The next example shows a particular situation where the previous proposition directly applies to compute Rees valuations and rational powers of certain classes of ideals.

\begin{remark}\label{rem:determine}
The equivalence (i) $\Leftrightarrow$ (iii) in \autoref{prop_min_vals} shows that if one knows the Rees valuations of an ideal, then one can compute its rational power; and conversely, if one knows the rational powers of the ideal, then one can determine its Rees valuations. 
\end{remark}

\begin{example}\label{example_subs_arr}
	Let $R$ be a regular ring and $I\subseteq R$ an ideal. Assume that there exist prime ideals $\frp_1,\ldots, \frp_r$ and positive integers $a_1, \ldots, a_r$ such that $\overline{I^n}=\cap_{i=1}^r \frp_i^{(na_i)}$ is the primary decomposition of $\overline{I^n}$ for every $n\in \NN$ that is irredundant for some $n$; here $\frp_i^{(m)}$ denotes the $m$th-symbolic power of  $\frp_i$.  Since $R_{\frp_i}$ is a regular local ring for every $i$, the functions $v_i:R\to \NN$ given by $v_i( x)=\max\big\{n\mid x\in \frp_i^{(n)}\big\}$ extend to discrete valuations of $\Quot(R)$. Therefore, by \autoref{prop_min_vals} one has $\RV(I)=\{v_1,\ldots, v_r\}$ and  $\overline{I^w}=\cap_{i=1}^r \frp_i^{(\lceil wa_i\rceil)}$ for each  $w\in \QQ_{\gs 0}$ .
	
	For a particular example, consider $R$ a polynomial ring over a field and $H_1,\ldots, H_t$ subspaces of $R_1$. For each $i$ let  $I_i$ be the ideal generated by $H_i$ and set $I=I_1\cdots I_t$. The powers of $I$ are all integrally closed and  irredundant  primary decompositions  of them having the same format above can be found in \cite[Theorem 3.2]{CT22}. Therefore, these primary decompositions can be used to explicitly describe the Rees valuations and rational powers of $I$. 
\end{example}

The following proposition shows that for $w\gg 0$, every $\overline{I^w}$ can be expressed in terms of finitely many rational powers of $I$. 

\begin{proposition}\label{Noeth_rat_power}
Adopt \autoref{setup_Back_rat} and assume  $R$ is analytically unramified. 	Let  $\RV(I)=\{v_1, \dots, v_r\}$ and  set $e:=\lcm(v_1(I), \dots, v_r(I))$ (or any multiple of it). Let $u$ be a variable of degree 1, and consider the $\NN$-graded $R$-algebras
$$A:=R[Iu,Iu^2,\ldots, Iu^e]\subseteq B:=R[u].$$
If $C$ is the integral closure of $A$ in $B$, then

\begin{enumerate}[\rm (1)]
\item $C=\oplus_{n\in \NN}\overline{I^{\frac{n}{e}}}u^n$. 
\item $C$ is a finitely generated $A$-module and furthermore a finitely generated $R$-algebra.
\item There exists $r,m\in \ZZz$ such that $\overline{I^{mt}}=\left(\overline{I^{m}}\right)^t$ for every $t\in \NN$ and for every $w\gs \frac{r}{e}$ we have $\overline{I^{w+m}}=\overline{I^{w}}\,\overline{I^{m}}.
$
\end{enumerate}
\end{proposition}

\begin{proof}
We begin with (1). By \cite[Theorem 2.3.2]{huneke2006integral}, the algebra $C$ is $\NN$-graded and its $e$th-Veronese   is the integral closure of $A_e:=\oplus_{n\in \NN} I^n u^{en}$ in $B_e:=R[u^e]$, which is $C_e:=\oplus_{n\in \NN} \overline{I^n} u^{en}$ by \cite[Proposition 5.2.1]{huneke2006integral}. Write $C=\oplus_{n \in \NN}J_n u^n$ and fix $n$.  If $x\in J_n$, then  $x^e\in J_{en}=\overline{I^n}$, and so  $x\in \overline{I^{\frac{n}{e}}}$. On the other hand, if $x\in \overline{I^{\frac{n}{e}}}$, then $x^e\in \overline{I^n}$, and so $(xu^n)^e\in C$. Thus,  $xu^n$ is integral over $A$ and then $xu^n\in C$. It follows that $x\in J_n$.

We continue with (2). Note that $A=\oplus_{n\in \NN} I^{\lceil\frac{n}{e}\rceil} u^{n}$ and is Noetherian.
By Rees' theorem \cite{rees1961note} (see also \cite[Theorem 9.1.2]{huneke2006integral}), there exists $k\in \ZZz$ such that, for every integer $n\gs ek$, we have
$$\overline{I^{\frac{n}{e}}}
\subseteq \overline{I^{\lfloor\frac{n}{e}\rfloor}}
\subseteq I^{\lfloor\frac{n}{e}\rfloor -k+1}
\subseteq I^{\lceil\frac{n}{e}\rceil -k }.
$$
If $C_{\gs ek}:=\oplus_{n\gs ek}\overline{I^{\frac{n}{e}}}u^n$, then the latter inclusions show that $C_{\gs ek}\subset Au^{ek}$. Therefore, $C_{\gs ek}$, and so $C$, is a finitely generated $A$-module. In particular, $C$ is a finitely generated $R$-algebra.

Finally we prove (3). By (2) and \cite[Lemma 13.10]{GORTZ_WEDHORN}, there exists $m\in \ZZz$ such that $\overline{I^{mt}}=\left(\overline{I^{m}}\right)^t$ for every $t\in \NN$. Set
$C_m:=\oplus_{n\in \NN} \overline{I^{mt}} u^{emt}$. 
From (2), it follows that $C$ is a finitely generated $C_m$-algebra, and  since $C_m\subset C$ is an integral extension, by \cite[Lemma 2.1.9]{huneke2006integral}, $C$ is a finitely generated $C_m$-module. The statement now follows by considering $r\in \ZZz$ to be the largest degree of an element in a finite set of generators of $C$ as a $C_m$-module.
\end{proof}

We continue with the following statement that allows for base change in rational powers (cf. \cite[Corollary 19.5.2]{huneke2006integral}). A ring homomorphism $f:R\to S$ is {\it normal} if it is flat and if for every $\frp\in \Spec(R)$ and field extension $\KK$ of $R_\frp/\frp R_\frp$ the ring $\KK\otimes_R S$ is normal.

\begin{proposition}\label{base_change}
	Adopt \autoref{setup_Back_rat}. Let $f:R\to S$ be a normal homomorphism with $S$  Noetherian. For every $w\in \QQez$ we have $\overline{I^w}S=\overline{(IS)^w}.$
\end{proposition}
\begin{proof}
	Let $e_1:=\lcm\left(\{v(I)\mid v\in \RV(I)\}\right)$, $e_2:=\lcm\left(\{v(I)\mid v\in \RV(IS)\}\right)$, and $e:=\lcm(e_1,e_2)$. Let $u, A, B, C$ be as in \autoref{Noeth_rat_power} defined using the latter $e$. 
	By \cite[Theorem 19.5.1]{huneke2006integral}, $C\otimes_R S$ is the integral closure of $A\otimes_R S$ in  $B\otimes_R S$. 
	Under the isomorphism $B\otimes_R S\cong S[u]$, $A\otimes_R S$ corresponds to $S[(IS)u, (IS)u^2,\ldots, (IS)u^e]$ and $C\otimes_R S$ to $\oplus_{n\in \NN}(\overline{I^{\frac{n}{e}}}S)u^n$. The result follows by applying \autoref{Noeth_rat_power}(1) to $IS$.
\end{proof}

\subsection*{Rational powers of ideal sheaves} In this subsection we introduce the notion 
of rational powers of ideal sheaves in normal varieties over algebraically closed fields, inspired by \cite[D\'efinition 4.2.2]{lejeune2008cloture} and \cite[Definition 9.6.2]{LAZARSFELD}. 

\begin{setup}\label{setup_setup}
Let $X$ be a normal variety (integral, separated, and of finite type) over an algebraically closed field $\kk$ and let $\I$ be a nonzero $\O_X$-ideal.
\end{setup}

\begin{definition}[Rational powers of ideal sheaves]\label{def_int_clo}
Adopt \autoref{setup_setup}. Let $v: X^+\longrightarrow X$ be the normalization of the blowup of $X$ along $\I$, and let $E$ be its exceptional divisor. Note that $\I\cdot \O_{X^+}=\O_{X^+}(-E)$.
The {\it integral closure} of $\I$, denoted by $\overline{\I}$, is the $\O_X$-ideal $\overline{\I}:=v_*\left(\O_{X^+}(-E)\right)$ (cf. \cite[Definition 9.6.2]{LAZARSFELD}).
Moreover, for  $w\in \QQ_{\gs 0}$, we call the  $\O_X$-ideal $$\overline{\I^w}:=v_*\left(\O_{X^+}(-\lceil wE\rceil)\right)$$ the {\it $w$-th rational power} of $\I$.
\end{definition}

\begin{remark}\label{global}
	If $X$ in \autoref{def_int_clo} is affine, so that $\I = \tilde{I}$ for some $\kk[X]$-ideal $I$, then from \autoref{lemma_same_vals} it follows that $\overline{I^w}=\Gamma(X,\overline{\I^w})$ are the rational powers of $I$ as in \autoref{def_int_clo_rati}.
\end{remark}

\begin{remark}[{cf. \autoref{discrete_local}}]\label{discrete}
Write $E = \sum_{i=1}^r e_iE_i$, where $E_1,\ldots, E_r$ are distinct prime divisors of $X^+$, and set $e:=\lcm(e_1,\ldots, e_r)$ (or any multiple of it). 
Fix  $w\in \QQ_{>0}$ and let $n := \lceil we\rceil$. Then, $\overline{\I^w}= \overline{\I^{\frac{n}{e}}}$.   Indeed, since the statement is local we can assume $X$ is affine and that all the $E_i$ are supported in $X$. Fix $x\in \Gamma\left(X,\overline{\I^w}\right)$ so that  $\ord_{E_i}(v^*(x))\gs we_i$ for every $i$. Therefore, $\frac{e}{e_i}\ord_{E_i}(v^*(x))\gs n$ and it follows that  $\ord_{E_i}(v^*(x))\gs \frac{n}{e}e_i$.  We conclude $\overline{\I^w}\subseteq \overline{\I^{\frac{n}{e}}}$. Since the reverse inclusion clearly holds, we obtain $\overline{\I^w}= \overline{\I^{\frac{n}{e}}}$.
\end{remark}

We  have the following alternative way of computing rational powers similar to \cite[Remark 9.6.4]{LAZARSFELD} and \cite[Th\'eor\`eme 2.1]{lejeune2008cloture}. We use the definition of {\it projective morphism} from \cite{EGA}, which is more general than the one of \cite{Har} (see  \cite[tag 01W7, Lemma 29.43.3]{stacks-project}), and includes blowups of arbitrary varieties along  ideal sheaves \cite[tag 01OF, Lemma 31.32.13]{stacks-project}.

\begin{lemma}\label{alternative}
Adopt \autoref{setup_setup}.  If $v': X'\longrightarrow X$ is a projective birrational morphism such that $X'$ is a normal variety and $\I\cdot \O_{X'}=\O_{X'}(-F)$, for some effective Cartier divisor $F$ on $X'$, then  $$\overline{\I^w}=v'_*\left(\O_{X'}(-\lceil wF\rceil)\right).$$
\end{lemma}
\begin{proof}
 By the universal properties of blowups and normalization \cite[Proposition II.7.14, Exercise II.3.8]{Har}, the map $v'$ factors through the normalized blowup $v: X^+\longrightarrow X$.    
Let $g:X'\longrightarrow X^+$ be the factoring map. 
Let $E$ be such that $\I\cdot \O_{X^+}=\O_{X^+}(-E)$. 
We have the following 
isomorphism induced by $g^{*}\left(\O_{X^+}\right)=g^{-1} \O_{X^+}\otimes_{g^{-1},  \O_{X^+}}\O_{X'}\xrightarrow{\,\,\,\,\sim\,\,\,\,} \O_{X'}$
\begin{equation}\label{eqn_map1}
g^*(\O_{X^+}(-E))= g^{-1}(\O_{X^+}(-E))\otimes_{g^{-1} \O_{X^+}}\O_{X'}
\xrightarrow{\,\,\,\,\sim\,\,\,\,}
\O_{X^+}(-E)\cdot\O_{X'}=\I\cdot\O_{X'} =  \O_{X'}(-F).
\end{equation}
Indeed, a surjection of  invertible sheaves is an isomorphism. Since $g$ is also projective \cite[tag 01W7, Lemma 29.43.15]{stacks-project} we have  $g_{*}\left(\O_{X'}\right)=\O_{X^+}$ by \cite[proof of Corollary III.11.4]{Har} and \cite[tag 01W7, Lemma 29.43.4]{stacks-project}. Therefore, for every $n\in \ZZz$ we have the following commutative diagram of natural maps.
\begin{equation*}
	\begin{tikzcd}
		\O_{X^+} \arrow{r}{\sim}  & g_*g^*\left(\O_{X^+}\right) \arrow{r}{\sim}  & g_*\left(\O_{X'}\right) = \O_{X^+}\\
		\arrow[u,hook] \O_{X^+}(-nE) \arrow{r}{\sim} & \arrow[u] g_*g^*\left(\O_{X^+}(-nE)\right) \arrow[r,"\sim","\autoref{eqn_map1}"'] &\arrow[u,hook]  g_*\left(\O_{X'}(-nF)\right)
	\end{tikzcd}
\end{equation*}
It follows that $\O_{X^+}(-nE) =g_*\left(\O_{X'}(-nF)\right)$, and this implies  $v'_*\left(\O_{X'}(-nF)\right)=
v_*\left(\O_{X^+}(-nE)\right)=
\overline{\I^n}$ for every $n\in \NN$.
Set $F = \sum_{i=1}^r e_iF_i$, where $F_1,\ldots, F_r$ are distinct prime divisors of $X'$. Working on affine subsets of $X$, we have shown that the integral closures of the powers of $\I$ are defined via the valuations $\ord_{F_i}$. Therefore, \autoref{lemma_same_vals} implies that the same is true for rational powers of $\I$, completing the proof.
%
\end{proof}

The previous lemma yields  the following version of the birrational transformation rule for rational powers of ideals  (cf. \cite[Theorem 9.2.33]{LAZARSFELD}, \cite[Lemme 4.15]{lejeune2008cloture}).

\begin{theorem}[Birrational transformation rule for rational powers]\label{birr_transf}
Adopt \autoref{setup_setup}. Let $f:Y\longrightarrow X$ be a projective birational morphism with $Y$  a normal variety and  set $\I_Y:=\I\cdot \O_Y$.   For every $w\in \QQ_{\gs 0}$ we have
$$\overline{\I^w}=f_*(\overline{\I^w_Y}).$$
\end{theorem}
\begin{proof}
Let $\mu: Y^+\longrightarrow Y$ be the normalization of the blowup of $Y$ along $\I_Y$ and $D$ its exceptional divisor so that  $(f\mu)^{-1}\I\cdot \O_{Y^+}= \mu^{-1}\I_Y\cdot \O_{Y^+}=\O_{Y^+}(-D)$. Then, by \autoref{alternative} we have
$$\overline{\I^w}=f_*\mu_*\left(\O_{Y^+}(-\lceil wD\rceil)\right)=f_*(\overline{\I^w_Y}),$$
as desired.
\end{proof}

\section{Rees valuations and rational powers of invariant ideals}\label{sec_Rees_Rat}

In this section we describe the Rees valuations and rational powers of several classes of invariant ideals. This will prove \autoref{thmB}. 
We begin with the following setup.

\begin{setup}\label{setup_Invar}
Let $\kk$ be a field and let $R$ be a finitely generated  $\kk$-algebra.  Assume $R$ is an integral domain and denote by $\Quot(R)$ its field of fractions. We set $I\subset R$ to be an ideal and  $\B:=\{\fb_\ell\mid \ell\in \Lambda\}\subseteq R$ a $\kk$-basis of $R$.
\end{setup}

The following definitions will be used in the main results of this section.

\begin{definition}\label{def_B_monomial}
Adopt \autoref{setup_Invar}.
 We say that $v: R\rightarrow \NN$ is a  {\it $\B$-monomial function} (on $R$) if for every  $f=\sum_{c_i\in \kk^*}c_i \fb_{\ell_i}\in R$ we have  $v(f)=\min\{v\left( \fb_{\ell_i}\right)\}$. If in addition  $v$ is a discrete  valuation  of  $R$, then we say that it  is a {\it $\B$-monomial valuation}.
%
%
\end{definition}

\begin{definition}\label{def_pos_poly}
We say that a polyhedron $\Gamma\subseteq \RR^d_{\gs 0}$ is {\it positive} if it is of the form $\Gamma=\mathscr{Q}
+\RR_{\gs 0}^d$ with 
$\mathscr{Q}\subseteq \RR_{\gs 0}^d$ a rational polytope.
We say that a (supporting) hyperplane $H$ of $\Gamma$ is  {\it non-coordinate} if it is  defined by $\langle\bh,\fX\rangle=c$ where $\bh=(h_1,\ldots, h_d)\in \NN^d$, $c\in \ZZ_{>0}$, and $\fX=(X_1,\ldots, X_d)$ are the coordinate variables of $\RR^d$.
\end{definition}

The following proposition provides a description of the Rees valuations and rational powers of ideals satisfying a natural assumption.

\begin{proposition}\label{prop_Rees_Package}
	Adopt  \autoref{setup_Invar}.
	Let  $v_1, \ldots, v_d$ be $\B$-monomial functions  of $R$ and set $\underline{v}=(v_1,\ldots, v_d)$.
 Let $\Gamma\subseteq \RR^d_{\gs 0}$ be a positive polyhedron such that
for every $n\in \NN$ the integral closure $\overline{I^n}$ is spanned as a $\kk$-vector space by the set
$\{\fb\in \B\mid \underline{v}(\fb)\in n\Gamma\}$.
 Let $H_1, \ldots, H_r$ be the   non-coordinate 
 hyperplanes of $\Gamma$  where $H_k$ is defined by $\langle\bh_k,\fX\rangle=c_k\in \ZZz$  and  $\bh_k=(h_{k,1},\ldots, h_{k,d})\in \NN^d$. Assume that for each $k$ the $\B$-monomial function $V_k:=\langle\bh_k,\underline{v}\rangle = h_{k,1}v_1+\ldots+h_{k,d}v_d$ is a valuation of $R$. Then we have
\begin{enumerate}[\rm (1)]
	\item $\RV(I)=\{V_1,\ldots, V_r\}$.
	\item For every $w\in \QQez$, $\overline{I^w}$ is spanned as a $\kk$-vector space by  $\{\fb\in \B\mid \underline{v}(\fb)\in w\Gamma\}.$
\end{enumerate}
\end{proposition}
\begin{proof}
Part (1) follows from 
 \autoref{prop_min_vals}. Part (2) follows from (1) since $f=\sum_{c_i\in \kk^*}c_i \fb_{\ell_i}\in \overline{I^w}$ if and only if $V_k(\fb_{\ell_i})\gs wc_k$ for every $i, k$ if and only if $\uv(\fb_{\ell_i})\in w\Gamma$ for every $i$.
\end{proof}

The previous proposition leads to the following notion.

\begin{definition}[Rees package]\label{def_Rees_Package}
	If $\B$, $\underline{v}=(v_1,\ldots, v_d)$, and $\Gamma$ as in \autoref{prop_Rees_Package} exist, then we say that  $(\B, \underline{v}, \Gamma)$ is a {\it Rees package} for $I$.
\end{definition}

The following lemma offers a sufficient condition for the functions $V_k$ in \autoref{prop_Rees_Package} to be valuations.  

\begin{lemma}\label{lemma_Rees_package}
	Adopt  \autoref{setup_Invar}. 
	Let $v_1,\ldots, v_d$ be $\B$-monomial valuations and set $\underline{v}:=(v_1,\ldots, v_d)$.  
	Assume that 
	for every $f=\sum_{c_i\in \kk^*}c_i \fb_{\ell_i}$ and $g=\sum_{e_i\in \kk^*}e_i \fb_{\mu_i}$ in $R$ with $\underline{v}(\fb_{\ell_i})=\underline{v}(\fb_{\ell_j})$ and $\underline{v}(\fb_{\mu_i})=\underline{v}(\fb_{\mu_j})$ for every $i,j$, there exist $\fb\in\B$ in the expansion of $fg$ as a linear combination of $\B$ such that $\uv(\fb)=\uv(f)+\uv(g)$.
	%
	Then for every nonzero $\bh=(h_{1},\ldots, h_{d})\in \NN^d$ the  $\B$-monomial function $V:=\langle\bh,\underline{v}\rangle = h_{1}v_1+\ldots+h_{d}v_d$ is a valuation of $R$.
\end{lemma}
\begin{proof}
	Fix $f=\sum_{c_i\in \kk^*}c_i \fb_{\ell_i}$ and $g=\sum_{e_i\in \kk^*}e_i \fb_{\mu_i}$ in $R$. Since  $V$ is $\B$-monomial, we clearly have $V(f+g)\gs \min\{V(f), V(g)\}$.
	Notice that 
	$V(fg) = V(f) + V(g)$ is equivalent to  $V(\fb)=V(f)+V(g)$ for some $\fb\in \B$ in the expansion of $fg$. 
	To show the existence of such $\fb$ we can further assume
	that for every $i$ one has  $V(\fb_{\ell_i})=V(f)$ and $V(\fb_{\mu_i})=V(g)$. 
	Let  $\mathscr{P}_f:=\conv\left(\{\uv(\fb_{\ell_i})\}\right)$ and
	$\mathscr{P}_g:=\conv\left(\{\uv(\fb_{\mu_i})\}\right)$ and notice that these polytopes belong to the parallel hyperplanes $\langle\bh,\fX\rangle=V(f)$ and $\langle\bh,\fX\rangle=V(g)$, respectively. Consider a vertex $\ba$ of the Minkowski sum  $\cP_f+\cP_g$. Therefore, there exist  unique  $\ba_f\in \cP_f\cap \NN^d$ and $\ba_g\in \cP_g\cap \NN^d$ such that $\ba_f+\ba_g=\ba$, and then 
	we can assume $\uv(\fb_{\ell_i})=\ba_f$ and $\uv(\fb_{\mu_j})=\ba_g$ for every $i,j$. The statement now follows by the  assumption.
\end{proof}

\subsection{Torus-invariant (monomial) ideals in semigroup rings}

In this subsection we show that Torus-invariant ideals, i.e. those generated by monomials, in semigroup rings have a Rees package.

\begin{setup}\label{setup_torus_inv}
Let $\mathscr{S}\subset \ZZ^s$ be an \emph{affine semigroup} of rank $s$.  Let $\kk$ be a field and let $\kk[\cS]$ be the {\it semigroup ring}  of $\cS$, i.e., the subalgebra of the Laurent polynomial ring $\kk[x_{1}^{\pm 1},\dots,x_{s}^{\pm 1}]$ spanned as a vector space by the set of monomials $\fx^\cS:=\{\fx^\bn=x_1^{n_1}\cdots x_s^{n_s}\mid \bn=(n_1,\ldots, n_s)\in \cS\}$, and multiplication induced by the operation of $\cS$. Let  $\C(\cS):=\RR_{\gs 0} \cS$ be the cone generated by $\cS$ and assume it is {\it strongly convex}, i.e., $\C(\cS)$ does not contain any subspace of positive dimension.
Let   $\F_1, \ldots, \F_d $  be the facets of $\C(\cS)$ and  for each $k$ let $\F_k$ be defined by $\langle\mathbf{f}_k,\fX \rangle=0$ with $\mathbf{f}_k\in \ZZ^s$.  Let $v_k$ be the $\fx^\cS$-monomial valuation determined by  $\fx^\bn\mapsto \langle\mathbf{f}_k, \bn \rangle$.

 Let $I\subseteq \kk[\cS]$ be a {\it monomial} ideal, i.e., it is generated by monomials. Set $\log(I):=\{\bn\in \cS\mid \fx^\bn\in I\}$ and $\Gamma(I):=\conv\left(\log(I)\right)\subseteq \RR^s$.
\end{setup}

We refer to \autoref{example_mon_sem} for a visual representation of the  previous notations. The following proposition is an important ingredient to describe the Rees package of monomial ideals.

\begin{proposition}\label{int.cls.mon.aff}
	Adopt \autoref{setup_torus_inv}. The integral closure $\overline{I}$ is  a monomial ideal generated by
	the monomials
  $	\{\fx^\bn\mid \bn\in \Gamma(I)\cap \cS\}$.
\end{proposition}
\begin{proof}
Let $\tilde{\cS}:=\ZZ\cS\cap \RR_{\gs 0}\cS$ be the normalization of $\cS$.  By \cite[Theorem 4.46]{bruns2009polytopes} the statement holds for $I\,\kk[\tilde{\cS}]$. The statement now follows for $I$ by \cite[Proposition 1.6.1]{huneke2006integral}  since $\kk[\cS]\subseteq \kk[\tilde{\cS}]$ is an integral extension \cite[Lemma 4.41]{bruns2009polytopes} and $\Gamma(I\,\kk[\tilde{\cS}])=\Gamma(I)$.
\end{proof}

We are now ready to describe a Rees package for monomial ideals.

\begin{theorem} \label{mon.aff.rees}
	Adopt \autoref{setup_torus_inv} .  If $\uv:=(v_1,\ldots, v_d)$ and $\Sigma:=\conv\left(\uv\left(\log(I)\right)\right)\subseteq \RR_{\gs 0}^d$,   then $(\fx^\cS,\uv, \Sigma)$ is a Rees package for $I$. In particular, the Rees valuations of $I$ are determined by the non-coordinates hyperplanes of $\Sigma$ and  for every $w\in \QQez$, $\overline{I^w}$ is spanned as a $\kk$-vector space by    $\{\fx^\bn\in \fx^\cS\mid \underline{v}\left(\fx^\bn\right)\in w\Sigma\}=\{\fx^\bn\mid \bn\in w\Gamma(I)\cap \cS\}$.
\end{theorem}
\begin{proof}
Clearly $\Gamma(I^n)=n\Gamma(I)$ for every $n\in \NN$, and so $n\Sigma=\conv\left(\uv\left(\log(I^n)\right)\right)$.  Also, for every $\bmm\in \NN^d$ the set $\uv^{-1}(\bmm)\cap \fx^\cS$ has at most one element, therefore  the assumption in \autoref{lemma_Rees_package}  is trivially satisfied  by  $\fx^\cS$ and $\underline{v}$.  The statements then follow from \autoref{int.cls.mon.aff}, 
 \autoref{prop_Rees_Package}, and \autoref{def_Rees_Package}.
\end{proof}

\begin{remark}\label{rem_NND}
We note that Theorem \autoref{mon.aff.rees} applies to any ideal whose integral closure is monomial. Such ideals are called {\it Newton non-degenerate} in the literature; see \cite{Saia96}.
\end{remark}

In the next example we describe the notation in this section for one specific ideal.

\begin{example}\label{example_mon_sem}
We follow the notation from \autoref{mon.aff.rees}. Let $\cS\subset \ZZ^2$ be the (normal) semigroup generated by $\{(2,1), (1,3)\}$, see \autoref{fig1} below. Thus $\kk[\cS]=\kk[x_1^2x_2, x_1x_2^3]$. The valuations $v_1, v_2$ corresponding to the rays $\F_1, \F_2$ are determined by  $\fx^\bn \mapsto -n_1+2n_2$ and $\fx^\bn \mapsto 3n_1-n_2$, respectively. The ideal $I\subset \kk[\cS]$ is generated by $x_1^4x_2^2$ and $x_1^3x_2^4$. Therefore $\Gamma(I)$ has as facets the rays $(4,2)+\RR_{\gs 0}(2,1)$, $(3,4)+\RR_{\gs 0}(1,3)$, and the segment $\conv\left((3,4), (4,2)\right)$. It follows that $\Sigma$ has supporting hyperplanes  $H: v_1=0$,  $H_1: v_1+v_2=10$, and $H_2: v_2=5$. The latter two, $H_1,H_2$, are the non-coordinate ones and so these are the ones that determine the Rees valuations of $I$.

	\begin{figure}[ht]
	\centering
	\mbox{\subfigure{\includegraphics[height=1.6in]{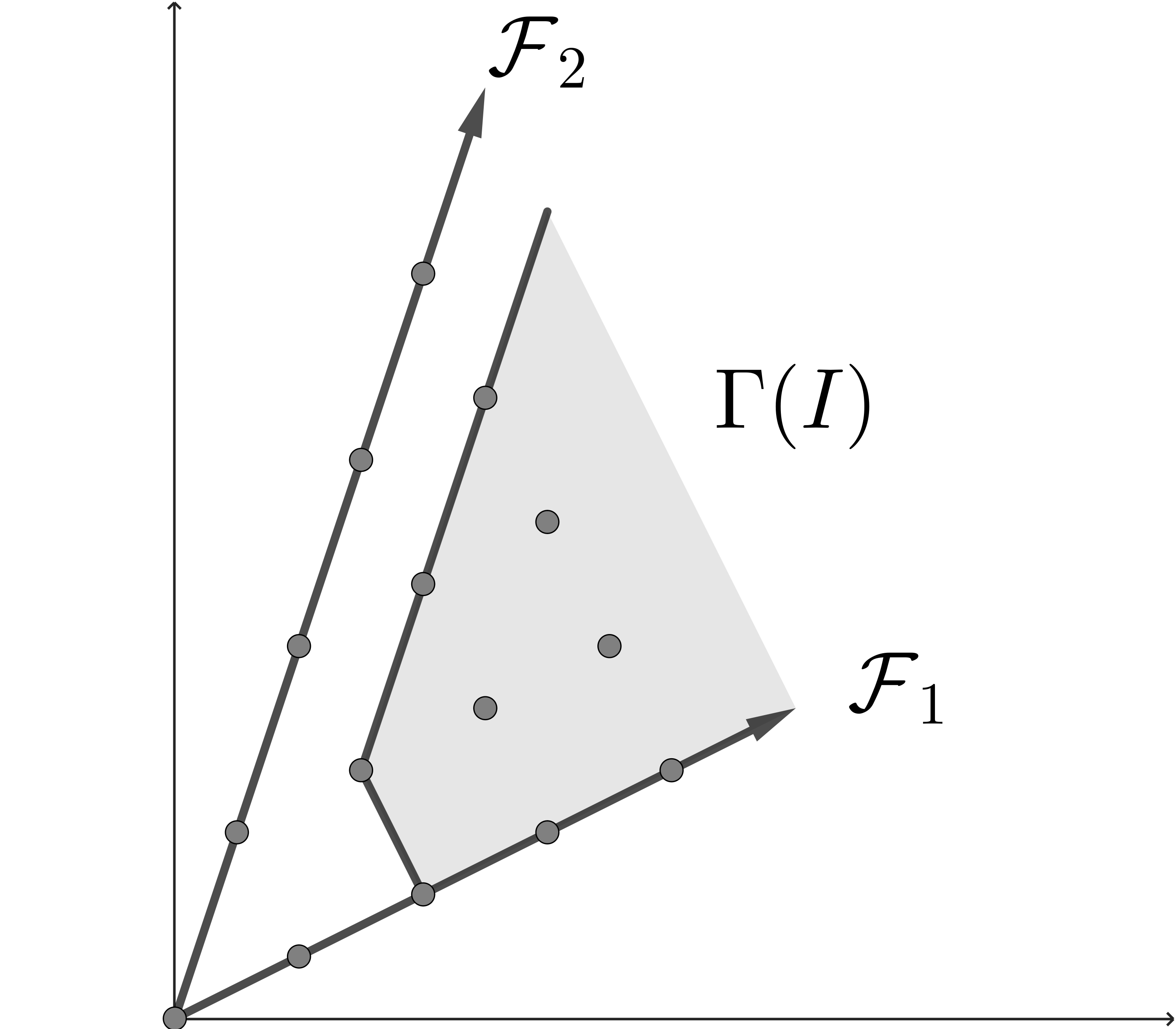}} \subfigure{\includegraphics[height=1.6in]{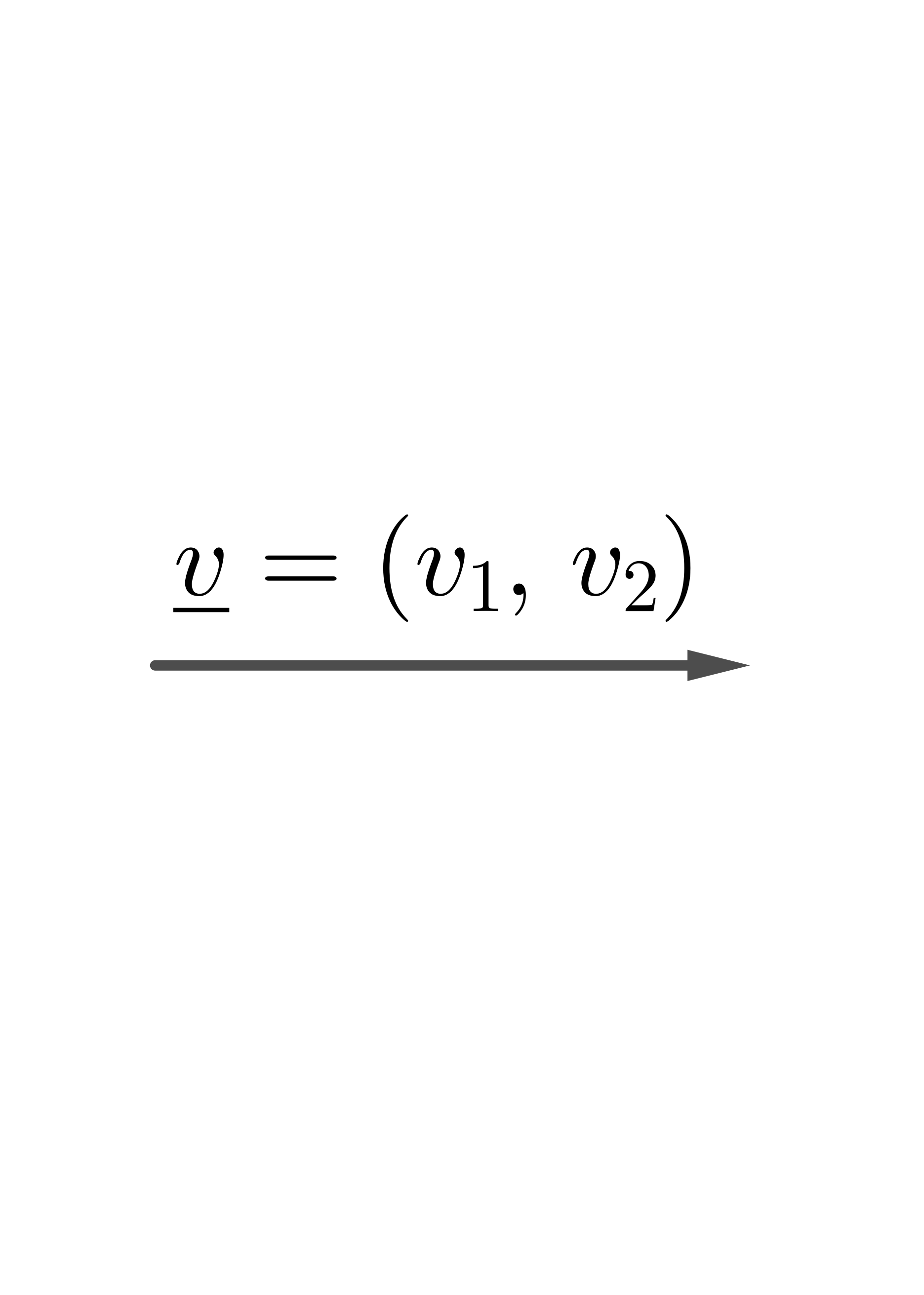}}
		\subfigure{\includegraphics[height=1.6in]{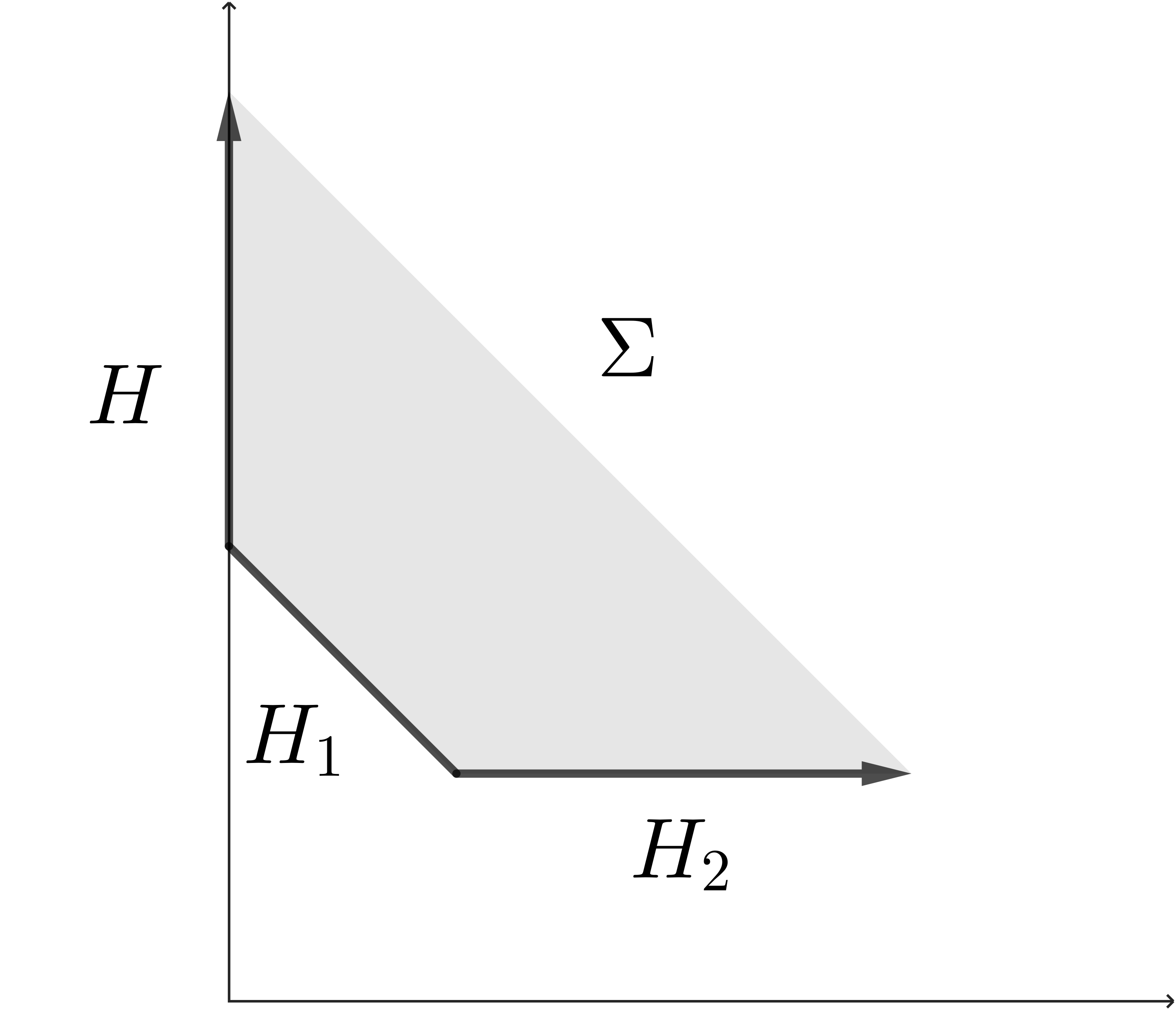} }
	}
	\caption{The semigroup $\cS= \NN(2,1)+ \NN(1,3)$ and ideal $I=\left(x_1^4x_2^2,\, x_1^3x_2^4\right)\subset \kk[\cS]$.} \label{fig1}
\end{figure}
\end{example}
\subsection{Sums of products of determinantal ideals and $\GL_m(\CC) \times \GL_n(\CC)$-invariant ideals}\label{subsect_gen} In this subsection we show that ideals obtained by taking sums of products of determinantal ideals in a generic matrix $X$ have a Rees package. As a consequence we obtain that $\GL_m(\CC) \times \GL_n(\CC)$-invariant ideals of the polynomial ring $\CC[X]$ have a Rees package. Throughout we follow the treatment  in \cite{deConcini80,bruns2022determinants,Henriques_Varbaro,Bruns91}.

\begin{setup}\label{setup_gen}
Let $X$ be an $m\times n$ generic matrix with $m\ls n$. Let $\kk$ be a field and $R=\kk[X]$ the corresponding polynomial ring. For any $t\ls m$ let $I_t\subset R$ denote the ideal generated by the $t$-minors of $X$. The $t$-minor corresponding to  rows $i_1< \cdots <i_t\ls m$ and columns $j_1< \cdots < j_t\ls n$ will be denoted as  $[i_1,\dots, i_t | j_1, \dots, j_t]$. For an positive integer $p$ we denote by $[p]$ the set $[p]:=\{1,\ldots, p\}$. A {\it diagram} $\sigma$ is a vector $\sigma:=(s_1,\ldots, s_p)\in \NN^p$ with $s_1\gs \cdots \gs s_p$. If $\tau=(t_1,\ldots, t_q)$ and $\sigma=(s_1,\ldots, s_p)$ are two diagrams,  we write $\tau\subseteq \sigma$ if $q\ls p$ and $t_i\ls s_i$ for every $i\in [q]$.   Let $\H_m$ be the set of all  diagrams such that $s_i\ls m$ for every $i$. For $t\in [m]$ we denote by $\gamma_t$ the function $\gamma_t: \H_m\to \NN$ given by  $\gamma_t(s_1,\ldots, s_p):=\sum_{i=1}^p \max\{0,s_i-t+1\}$.
\end{setup}


We define a partial order on the minors of $X$ in the following way:
$$ [i_1,\dots, i_t | j_1, \dots, j_t] \leqslant [e_1,\dots, e_h | f_1, \dots, f_h] \iff t\geqslant h \text{ and }i_q \leqslant e_q, \,j_q \leqslant~ f_q~ \text{ for every } q \in [h].\label{eqn.minor.order}$$
A {\it standard monomial}  is a product of minors $\underline{\delta}:=\delta_1\cdots \delta_p$ with $\delta_1\leqslant \dots \leqslant \delta_p$.
For every $\underline{\delta}$ we associate a diagram $|\underline{\delta}|:=(s_1,\ldots, s_p)\in \H_m$ where $\delta_i$ is a $s_i$-minor and call $\sigma$ the {\it shape} of $\underline{\delta}$.  Let $\B$ be the set of all standard monomials. By \cite[Theorem 2.1]{deConcini80} $\B$ forms a $\kk$-basis for $R$.
For each diagram $\sigma=(s_1,\ldots, s_p)\in \H_m$ we denote by $I_\sigma$ the product of determinantal ideals $I_\sigma:=I_{s_1}\cdots I_{s_p}$, and for any subset $\Lambda\subseteq \H_m$ we set $$I_\Lambda:=\sum_{\sigma\in \Lambda}I_\sigma.$$
For each $t\in [m]$ we extend $\gamma_t$ to a  $
\B$-monomial valuation induced by
$\gamma_t(\underline{\delta}):=\gamma_t(|\underline{\delta}|)$.
Indeed these are valuations by \cite[Theorem 7.1]{deConcini80}. 

We now show that the ideals $I_\Lambda$ have a Rees package. In the following,  by $I^{(n)}$ we denote the $n$th-symbolic power of the ideal $I$.
\begin{theorem} \label{gen_det_rees}
Adopt \autoref{setup_gen} and the notation above. Let $\Lambda\subseteq \H_m$.   If $\underline{\gamma}:=(\gamma_1,\ldots, \gamma_m)$ and $\Gamma:=\conv\big(\underline{\gamma}(\Lambda)\big)+\RR_{\gs 0}^m$,   then $(\B,\underline{\gamma}, \Gamma)$ is a Rees package for $I_\Lambda$. In particular, the Rees valuations of $I_\Lambda$ are determined by the non-coordinates hyperplanes of $\Gamma$ and  for every $w\in \QQez$, $\overline{I_\Lambda^w}$ is spanned as a $\kk$-vector space by  $\{\underline{\delta}\in \B\mid \underline{\gamma}(\underline{\delta})\in w\Gamma\}$ and $$\overline{I_\Lambda^w}=\sum_{(a_1, \dots, a_m)\in w\Gamma\cap \NN^m} \left( \bigcap_{i=1}^m I_i^{(a_i)} \right).$$
\end{theorem}
\begin{proof}
By \cite[Theorems 2.2 \& 2.3]{Henriques_Varbaro}  
 the integral closure $\overline{I_\Lambda^n}$ is spanned as a $\kk$-vector space by the set
$\{\underline{\delta}\in \B\mid \underline{\gamma}(\underline{\delta})\in n\Gamma\}$.
By \cite[Proposition 7.4.2, Equations (7.2) \& (7.3)]{bruns2022determinants} the assumption in \autoref{lemma_Rees_package}  is satisfied  by $\B$ and $\underline{\gamma}$; indeed, notice that by \cite[Remark 3.2.9]{bruns2022determinants} we may assume $\kk=\overline{\kk}$,  and under this assumption  Segre products of domains are domains. 
Therefore,   $(\B, \underline{\gamma}, \Gamma)$ is a Rees package for $I_\Lambda$  by \autoref{prop_Rees_Package} and \autoref{def_Rees_Package}. 
The last formula follows from  \cite[Theorems 2.2 \& 2.3]{Henriques_Varbaro}.
\end{proof}

%
%
%

In the next example we illustrate the previous theorem. 


\begin{example}\label{example_det}
	Let $X$ be a $2\times 3$ generic matrix and $\Lambda=\{(2),\, (1,1,1 )\}$, i.e., $I_\Lambda=I_2+ I_1^3$. Since $\underline{\gamma}(2)=(2,1)$ and $\underline{\gamma}(1,1,1)=(3,0)$  we have $\Gamma=\conv\left((2,1), (3,0)\right)+\RR^2_{\gs 0}$. The non-coordinate 
	hyperplanes of $\Gamma$ are $X_1=2$ and $X_1+X_2=3$, 
	see \autoref{fig2}.
	 Therefore the Rees valuations of $I_\Lambda$ are  $\gamma_1$ and $\gamma_1+\gamma_2$. Finally, for every $w$ we have $$\overline{I_\Lambda^w}=
	\sum_{\stackrel{a_1,a_2\in \NN}{a_1\gs 2w, a_1+a_2\gs 3w}}
	I_1^{(a_1)}\cap I_2^{(a_2)}.$$
	\begin{figure}[hbt!]
		\centering
		\subfigure{\includegraphics[height=1.6in]{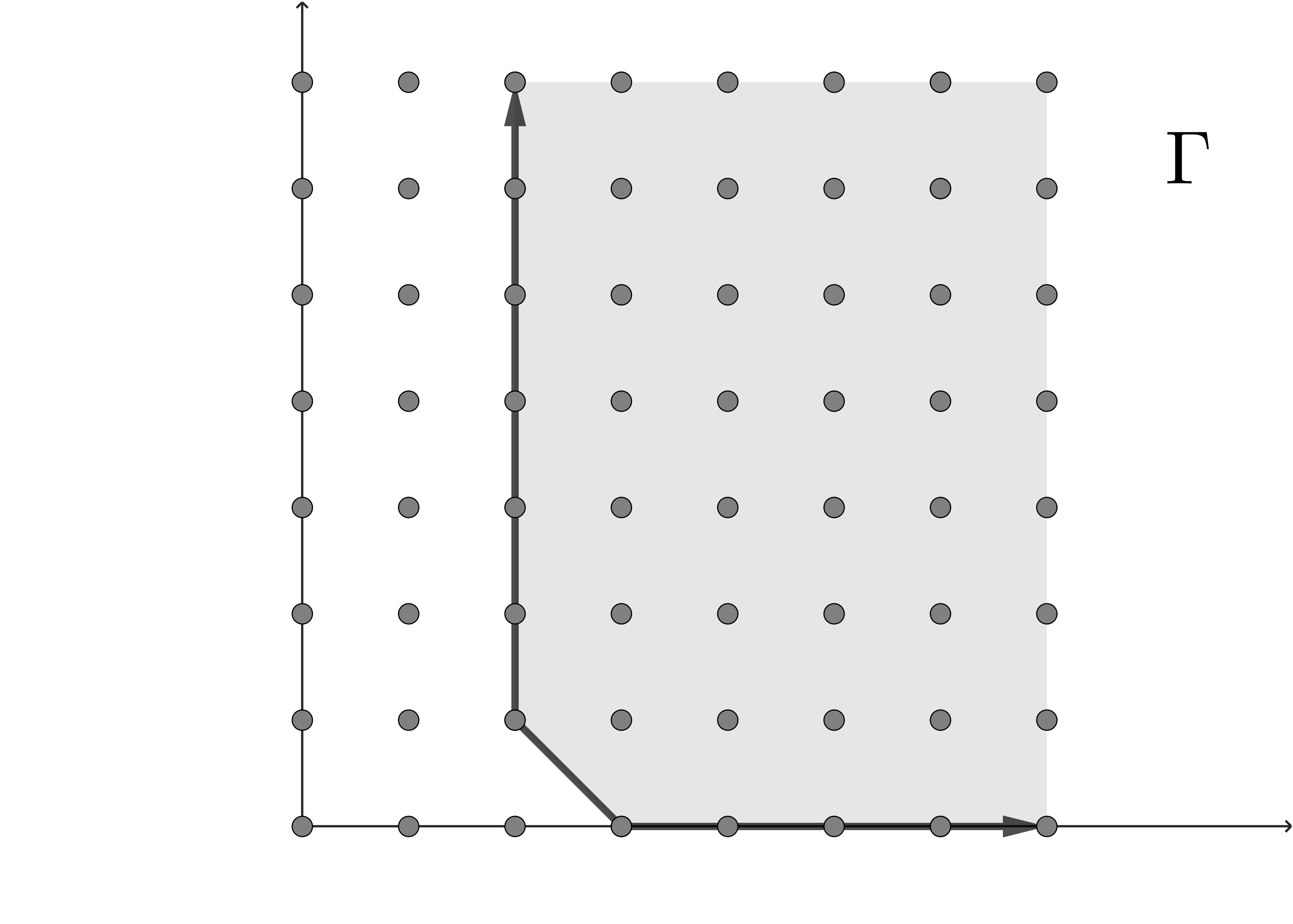}}
		\caption{The polyhedron $\Gamma=\conv\big(\underline{\gamma}(\Lambda)\big)+\RR_{\gs 0}^2=\conv\left((2,1), (3,0)\right)+\RR^2_{\gs 0}$.} \label{fig2}
	\end{figure}
\end{example}

\begin{example}\label{example_resurg}
As defined in  \cite{GHV13}, the {\it asymptotic resurgence} of a homogeneous ideal $I$ in a polynomial ring is  $\rho_a(I):=\sup\{s/r\mid I^{(st)}\not\subset I^{rt} \text{ for all }t\gg0\}$.  In \cite[Theorem 4.10]{DiPasq_et_al_19} the authors showed  that $\rho_a(I)=\max_{v\in \RV(I)}\left\{v(I)/\widehat{v}(I)  \right\}$,
where $ \displaystyle\widehat{v}(I)=\lim_{r\to \infty}v(I^{(r)})/r$. By  \autoref{gen_det_rees}, if $I_t$ is the determinantal ideal defined in \autoref{setup_gen}, then $\RV(I_t)=\{\gamma_1,\ldots, \gamma_t\}$. Using \cite[Theorem 7.1]{deConcini80} and \cite[Remark 7.23]{de2021blowup} one shows that $\rho_a(I_t)=\dfrac{\gamma_1(I_t)}{\displaystyle\widehat{\gamma_1}(I_t)}=\dfrac{t(m-t+1)}{m}$.
\end{example}

For the rest of this subsection assume $\kk=\CC$ and set $G:=\GL_m(\CC) \times \GL_n(\CC)$. Consider the action of $G$ on $R=\CC[X]$ induced by $X_{i,j}\mapsto (AXB^{-1})_{i,j}$ for $(A,B)\in G$. In \cite[Theorem 5.1]{deConcini80} the authors described certain irreducible representations $L_\sigma$ and  $\prescript{}{\sigma}{L}$ of $\GL_m(\CC)$ and $\GL_n(\CC)$, respectively, indexed by $\sigma\in \H_m$,
and  they  showed  that an ideal of $R$ is invariant under the action of $G$ if and only if it is of the form
$$\I(\Lambda):=\bigoplus_{\stackrel{\tau \subseteq \sigma}{ \text{for some }\tau\in\Lambda}} L_\sigma \otimes_\CC  \prescript{}{\sigma}{L}$$
for some subset $\Lambda\subseteq \H_m$. Moreover, in \cite[Theorem 8.2]{deConcini80} it is shown that $\overline{\I(\Lambda)}=\overline{I_\Lambda}$. Thus we obtain the following corollary.

\begin{corollary}\label{cor_inv_gen}
Under the assumptions and notations in \autoref{gen_det_rees} with $\kk=\CC$, for every $\Lambda\subseteq \H_m$ the $G$-invariant ideal $\I(\Lambda)$ has a Rees package.
\end{corollary}

\subsection{Sums of products of determinantal ideals of symmetric matrices and $\GL_n(\CC)$-invariant ideals} In this subsection we show that ideals obtained by taking sums of products of determinantal ideals in a symmetric generic matrix $Y$ have a Rees package. As a consequence we obtain that $\GL_n(\CC)$-invariant ideals of the polynomial ring $\CC[Y]$ have a Rees package. Throughout we follow the treatment  in \cite{Abeasis80,Henriques_Varbaro,JEFFRIES_MONTANO_VARBARO}.

\begin{setup}\label{setup_sym}
	Let $Y$ be an $n\times n$ generic symmetric matrix. Let $\kk$ be a field and $R=\kk[Y]$ the corresponding polynomial ring. For any $t\ls n$ let $J_t\subset R$ denote the ideal generated by the $t$-minors of $Y$. We follow the notation introduced in \autoref{setup_gen} for minors and diagrams.
\end{setup}

A minor
$ [i_1,\dots, i_t | j_1, \dots, j_t]$ of $Y$ is {\it doset} if $i_q\ls j_q$ for every $q\in [t]$. We define a partial order on the doset minors in the following way:
$$ [i_1,\dots, i_t | j_1, \dots, j_t] \leqslant [e_1,\dots, e_h | f_1, \dots, f_h] \iff t\geqslant h \text{ and }j_q \leqslant e_q, ~ \text{ for every } q \in [h].\label{eqn.doset.order}$$
A {\it standard monomial}  is a product of doset  minors $\underline{\delta}:=\delta_1\cdots \delta_p$ with $\delta_1\leqslant \dots \leqslant \delta_p$.
 For every $\underline{\delta}$ we associate a diagram $|\underline{\delta}|:=(s_1,\ldots, s_p)\in \H_n$ where $\delta_i$ is a $s_i$-minor and call $\sigma$ the {\it shape} of $\underline{\delta}$.  Let $\B$ be the set of all standard monomials. By \cite[\S 4.2]{JEFFRIES_MONTANO_VARBARO} (see also \cite{Abeasis80}) $\B$ forms a $\kk$-basis for $R$.
For each $\sigma=(s_1,\ldots, s_p)\in \H_n$ we denote by $J_\sigma$ the product of determinantal ideals $J_\sigma:=J_{s_1}\cdots J_{s_p}$, and for any subset $\Lambda\subseteq \H_n$ we set $$J_\Lambda:=\sum_{\sigma\in \Lambda}J_\sigma.$$
For each $t\in [n]$ we extend $\gamma_t$ to a  $
\B$-monomial valuation induced by
$\gamma_t(\underline{\delta}):=\gamma_t(|\underline{\delta}|)$.
Indeed these are valuations by \cite[Theorem 5.1]{Abeasis80}, or \cite[Theorem 2.7]{Henriques_Varbaro}.

\begin{lemma}\label{lem:sym_domain}
Adopt \autoref{setup_sym} and the notation above. Set $\underline{\gamma}:=(\gamma_1,\ldots, \gamma_n)$. For every $\bn\in \NN^n$ let $V_\bn$ (resp. $V_{>\bn}$) be the $\kk$-subspace of $R$ spanned by the standard monomials $\underline{\delta}$ such that $\underline{\gamma}(\underline{\delta})=\bn$ (resp. $\underline{\gamma}(\underline{\delta})>\bn$ componentwise). Then the associated graded ring $\G=\oplus_{\bn\in \NN^n}V_\bn/V_{>\bn}$ is a domain.
\end{lemma}
\begin{proof}
Let $X$ be a generic $n\times n$ matrix and $S$ be the subalgebra of $\kk[X]$ generated by the products of minors $[1,\ldots, t_1|i_1,\ldots, i_{t_1}][1,\ldots, t_2|i_1,\ldots, i_{t_2}]$ for some $1\ls i_1,i_2\ls n$. One proceeds as in \cite[proof of Thorem 1.6]{baetica01} via \cite[Lemma 5.2 and proof of Lemma 5.3]{deConcini76} to show $\G\cong S$. The conclusion follows. 
\end{proof}

We now show that the ideals $J_\Lambda$ have a Rees package.
\begin{theorem} \label{gen_sym_rees}
	Adopt \autoref{setup_sym} and the notation above. Let $\Lambda\subseteq \H_n$.   If $\underline{\gamma}:=(\gamma_1,\ldots, \gamma_n)$ and $\Gamma:=\conv\big(\underline{\gamma}(\Lambda)\big)+\RR_{\gs 0}^n$,   then $(\B,\underline{\gamma}, \Gamma)$ is a Rees package for $J_\Lambda$. In particular, the Rees valuations of $J_\Lambda$ are determined by the non-coordinates hyperplanes of $\Gamma$ and  for every $w\in \QQez$, $\overline{J_\Lambda^w}$ is spanned as a $\kk$-vector space by  $\{\underline{\delta}\in \B\mid \underline{\gamma}(\underline{\delta})\in w\Gamma\}$ and $$\overline{J_\Lambda^w}=\sum_{(a_1, \dots, a_n)\in w\Gamma\cap \NN^n} \left( \bigcap_{i=1}^n J_i^{(a_i)} \right).$$
\end{theorem}
\begin{proof}
 By \cite[Theorems 2.6 \& 2.7]{Henriques_Varbaro}  
 the integral closure $\overline{J_\Lambda^n}$ is spanned as a $\kk$-vector space by the set
 $\{\underline{\delta}\in \B\mid \underline{\gamma}(\underline{\delta})\in n\Gamma\}$. 
 By \autoref{lem:sym_domain} the assumption in \autoref{lemma_Rees_package}  is satisfied   by $\B$ and $\underline{\gamma}$.  
 Therefore,   $(\B, \underline{\gamma}, \Gamma)$ is a Rees package for $J_\Lambda$  by \autoref{prop_Rees_Package} and \autoref{def_Rees_Package}. 
 The last formula follows again from  \cite[Theorems 2.6 \& 2.7]{Henriques_Varbaro}.  
\end{proof}

For the rest of this subsection assume $\kk=\CC$ and set $G:=\GL_n(\CC)$. Consider the action of $G$ on $R=\CC[Y]$ induced by $Y_{i,j}\mapsto (AYA^{t})_{i,j}$ for $A\in G$. In \cite{Abeasis80} the author described certain irreducible representations $L_\sigma$ of $G$ indexed by $\sigma\in \H_n$,
and  they  showed  that an ideal of $R$ is invariant under the action of $G$ if and only if it is of the form
$$\J(\Lambda):=\bigoplus_{\stackrel{\tau \subseteq \sigma}{ \text{for some }\tau\in\Lambda}} L_\sigma $$
for some subset $\Lambda\subseteq \H_n$. Moreover,  in \cite[Theorem 6.1]{Abeasis80} (see also \cite[Theorem 2.11]{Henriques_Varbaro}) it  is shown  that $\overline{\J(\Lambda)}=\overline{J_\Lambda}$. Thus we obtain the following corollary.

\begin{corollary}\label{cor_inv_sym}
	Under the assumptions and notations in \autoref{gen_sym_rees} with $\kk=\CC$, for every $\Lambda\subseteq \H_n$ the $G$-invariant ideal $\J(\Lambda)$ has a Rees package.
\end{corollary}

\subsection{Sums of products of  ideals of Pfaffians  and $\GL_n(\CC)$-invariant ideals} In this subsection we show that ideals obtained by taking sums of products of ideals of Pfaffians in a skew-symmetric generic matrix $Z$ have a Rees package. As a consequence we obtain that $\GL_n(\CC)$-invariant ideals of the polynomial ring $\CC[Z]$ have a Rees package. Throughout we follow the treatment  in \cite{Abeasis80Young,Henriques_Varbaro,JEFFRIES_MONTANO_VARBARO}.

\begin{setup}\label{setup_pff}
	Let $Z$ be an $n\times n$ skew-symmetric matrix of variables ($Z=-Z^t$). We follow the notation introduced in \autoref{setup_gen}  for minors and diagrams. Let $\kk$ be a field and $R=\kk[Z]$ the corresponding polynomial ring.  For any $t\ls \lfloor n/2\rfloor$ a minor
	of $Z$ of the form $ [i_1,\dots, i_{2t} | i_1, \dots, i_{2t}]$ is a square of a polynomial which is called a {\it $2t$-Pfaffian}. We will denote this Pfaffian as $ [i_1,\dots, i_{2t}]$. Let $P_{2t}\subset R$ be the ideal generated by the $2t$-Pfaffians of $Z$.
\end{setup}

 We define a partial order on Pfaffians in the following way:
$$ [i_1,\dots, i_{2t}] \leqslant [e_1,\dots, e_{2h}] \iff t\geqslant h \text{ and }i_q \leqslant e_q, ~ \text{ for every } q \in [2h].\label{eqn.Pf.order}$$
A {\it standard monomial} in this case is a product of Pfaffians $\underline{\delta}:=\delta_1\cdots \delta_p$ with $\delta_1\leqslant \dots \leqslant \delta_p$.
For every $\underline{\delta}$ we associate a diagram $|\underline{\delta}|:=(s_1,\ldots, s_p)\in  \H_{\lfloor n/2\rfloor}$ where $\delta_i$ is a $2s_i$-Pfaffian and call $\sigma$ the {\it shape} of $\underline{\delta}$.   Let $\B$ be the set of all standard monomials. By \cite[\S 4.3]{JEFFRIES_MONTANO_VARBARO} (see also \cite{Abeasis80Young}) $\B$ forms a $\kk$-basis for $R$.
For each $\sigma=(s_1,\ldots, s_p)\in \H_{\lfloor n/2\rfloor}$ we denote by $J_\sigma$ the product of  ideals of Pfaffians $P_\sigma:=P_{2s_1}\cdots P_{2s_p}$, and for any subset $\Lambda\subseteq \H_{\lfloor n/2\rfloor}$ we set $$P_\Lambda:=\sum_{\sigma\in \Lambda}P_\sigma.$$
For each $t\in  \big[\lfloor n/2 \rfloor\big]$ we extend $\gamma_t$ to a  $
\B$-monomial valuation induced by
$\gamma_t(\underline{\delta}):=\gamma_t(|\underline{\delta}|)$.
Indeed these are valuations by \cite[Theorem 5.1]{Abeasis80Young},  or \cite[Theorem 2.10]{Henriques_Varbaro}.

We now show that the ideals $P_\Lambda$ have a Rees package.
\begin{theorem} \label{gen_sym_pff}
	Adopt \autoref{setup_pff} and the notation above. Let $\Lambda\subseteq \H_{\lfloor n/2\rfloor}$.   If $\underline{\gamma}:=(\gamma_1,\ldots, \gamma_{\lfloor n/2\rfloor})$ and $\Gamma:=\conv\big(\underline{\gamma}(\Lambda)\big)+\RR_{\gs 0}^{\lfloor n/2\rfloor}$,   then $(\B,\underline{\gamma}, \Gamma)$ is a Rees package for $P_\Lambda$. In particular, the Rees valuations of $P_\Lambda$ are determined by the non-coordinates hyperplanes of $\Gamma$ and  for every $w\in \QQez$, $\overline{P_\Lambda^w}$ is spanned as a $\kk$-vector space by  $\{\underline{\delta}\in \B\mid \underline{\gamma}(\underline{\delta})\in w\Gamma\}$ and $$\overline{P_\Lambda^w}=\sum_{\left(a_1, \dots, a_{\lfloor n/2\rfloor}\right)\in w\Gamma\cap  \NN^{\lfloor n/2\rfloor}} \left( \bigcap_{i=1}^{\lfloor n/2\rfloor} P_{2i}^{(a_i)} \right).$$
\end{theorem}
\begin{proof}
 By \cite[Theorems 2.9 \& 2.10]{Henriques_Varbaro}  
the integral closure $\overline{P_\Lambda^n}$ is spanned as a $\kk$-vector space by the set
$\{\underline{\delta}\in \B\mid \underline{\gamma}(\underline{\delta})\in n\Gamma\}$.
By \cite[Theorem 1.6 and its proof]{baetica01}  (see   \cite[Lemma 6.2 \& proof of Theorem 6.5]{deConcini76} to remove the assumption on the characteristic)  the assumption in \autoref{lemma_Rees_package}  is satisfied   by $\B$ and $\underline{\gamma}$.  
Therefore,   $(\B, \underline{\gamma}, \Gamma)$ is a Rees package for $P_\Lambda$  by \autoref{prop_Rees_Package} and \autoref{def_Rees_Package}. 
The last formula follows again from  \cite[Theorems 2.9 \& 2.10]{Henriques_Varbaro}.  
\end{proof}

For the rest of this subsection assume $\kk=\CC$ and set $G:=\GL_n(\CC)$. Consider the action of $G$ on $R=\CC[Z]$ induced by $Z_{i,j}\mapsto (AZA^{t})_{i,j}$ for $A\in G$. In \cite{Abeasis80Young} the authors described certain irreducible representations $L_\sigma$ of $G$ indexed by $\sigma\in \H_{\lfloor n/2\rfloor}$,
and  they  showed  that an ideal of $R$ is invariant under the action of $G$ if and only if it is of the form
$$\P(\Lambda):=\bigoplus_{\stackrel{\tau \subseteq \sigma}{ \text{for some }\tau\in\Lambda}} L_\sigma $$
for some subset $\Lambda\subseteq \H_{\lfloor n/2\rfloor}$. Moreover,  in \cite[Theorem 6.2]{Abeasis80Young}  it is shown  that $\overline{\P(\Lambda)}=\overline{P_\Lambda}$. Thus we obtain the following corollary.

\begin{corollary}\label{cor_inv_Pf}
	Under the assumptions and notations in \autoref{gen_sym_pff} with $\kk=\CC$, for every $\Lambda\subseteq \H_{\lfloor n/2\rfloor}$ the $G$-invariant ideal $\P(\Lambda)$ has a Rees package.
\end{corollary}

\subsection{Products of determinantal ideals of Hankel matrices}
 In this subsection we show that products of determinantal ideals  of Hankel matrices have a Rees package. Throughout we follow the treatment  in \cite{Conca98Hankel}.

 \begin{setup}\label{setup_Hank}
 	Let $\kk$ be a field and $R=\kk[\underline{x}]=\kk[x_1,\ldots,x_n]$ a polynomial ring. Fix $c\in [n]$ and let
     $W$ be the $c\times (n-c+1)$ Hankel matrix
 $$
 W:=W_c=\begin{pmatrix}
 	x_1 & x_2 & x_3 &\dots &  x_{n-c+1}\\
 	x_2 & x_3 &  \dots & \dots & \dots \\
 	x_3 & \dots & \dots & \dots & \dots \\
 	\vdots& \vdots & \vdots & \vdots & \vdots\\
 	x_c & \dots & \dots & \dots & x_n
 \end{pmatrix}.
 $$
  For any $t\ls\min\{c,n-c+1\}$ let $H_t\subset R$ denote the ideal generated by the $t$-minors of $W$.
   The ideals $H_t$ only depends on $t$ and $n$ \cite[Corollary 2.2]{Conca98Hankel}.
    Note that the maximum value for $t$ is $\lfloor (n+1)/2\rfloor$.
  We follow the notation introduced in \autoref{setup_gen} for minors and diagrams.
 \end{setup}

A minor of $W$ of the form $[1, \dots, t| a_1,  \dots, a_t]$ is a {\it maximal minor} and it will be  denoted by $[a_1, \dots, a_t]$. A monomial $x_{a_1}\cdots x_{a_t}$ of $R$ is a {\it $<_1$-chain} if  $a_1 <_1 a_2 \dots <_1 a_s$, where $a<_1b $ means $ a+1 <b$.
There is a bijective map $\left\{<_1 \text{chains of } R \right\} \leftrightarrow \{ \text{maximal minors}\}$
defined by $\phi(x_{a_1}\dots x_{a_t})=[a_1, a_2-1, \dots, a_t-t+1]$ 
which induces an injective map $$\Phi: \{\text{monomials of } R \}\to \{\text{product of maximal minors}\};$$
see \cite{Conca98Hankel} for details.
The elements in the image $\B$ of
$\Phi$ are  the \emph{standard monomials} of $W$.
 By \cite[Theorem 1.2]{Conca98Hankel}  $\B$ forms a $\kk$-basis for $R$.
For every   standard monomial $\underline{\delta}:=\delta_1\cdots \delta_p$ with  $\delta_i$  a $s_i$-maximal-minor and $s_1\gs \cdots \gs s_p$,  we associate a diagram $|\underline{\delta}|:=(s_1,\ldots, s_p)\in \H_{\lfloor(n+1)/2\rfloor}$.
For each $\sigma=(s_1,\ldots, s_p)\in \H_{\lfloor(n+1)/2\rfloor}$ we denote by $H_\sigma$ the product of determinantal ideals $$H_\sigma:=H_{s_1}\cdots H_{s_p}.$$
For each $t\in [n]$ we extend $\gamma_t$ to a  $
\B$-monomial valuation induced by
$\gamma_t(\underline{\delta}):=\gamma_t(|\underline{\delta}|)$.
Indeed these are valuations by \cite[Theorem 3.8]{Conca98Hankel}.

We now show that the ideals $H_\sigma$ have a Rees package.

\begin{theorem} \label{gen_Hank}
	Adopt \autoref{setup_Hank} and the notation above.  Let $\sigma=(s_1,\ldots,s_p)\in  \H_{\lfloor(n+1)/2\rfloor}$.   If $\underline{\gamma}:=(\gamma_1,\ldots, \gamma_n)$ and
	$\Gamma:=\underline{\gamma}(\sigma)+\RR_{\gs 0}^n$,   then $(\B,\underline{\gamma}, \Gamma)$ is a Rees package for $H_\sigma$. In particular, the Rees valuations of $H_\sigma$ are determined by the non-coordinates hyperplanes of $\Gamma$ and  $\RV(H_\sigma)\subseteq \{\gamma_1,\ldots, \gamma_{s_1}\}$. Furthermore,   for every $w\in \QQez$, $\overline{H_\sigma^w}$ is spanned as a $\kk$-vector space by  $\{\underline{\delta}\in \B\mid \underline{\gamma}(\underline{\delta})\in w\Gamma\}$ and  $$\overline{H_\sigma^w}= \bigcap_{i=1}^{n } H_i^{(\lceil w\gamma_i(\sigma)\rceil)}.$$
\end{theorem}
\begin{proof}
	By \cite[Theorem 3.12]{Conca98Hankel}  
	the integral closure $\overline{H_\sigma^n}$ is spanned as a $\kk$-vector space by the set
	$\{\underline{\delta}\in \B\mid \underline{\gamma}(\underline{\delta})\in n\Gamma\}$.
	 Thus the  statements follow form  \autoref{prop_Rees_Package} and \autoref{def_Rees_Package}. 
	The last formula follows  from  \cite[Theorems 3.8 \& 3.12]{Conca98Hankel}.  
\end{proof}

\section{Rees package and the summation formula}\label{sec_package_summ}

This section is devoted to the proof of \autoref{main_Rees_P}, which establishes that two ideals with a Rees package satisfy the summation formula for rational powers for every $w\in \QQ_{\gs 0}$. From this theorem \autoref{thmC} (or \autoref{cor_invs_all}) directly follows. We will continue using the notations and definitions in \autoref{sec_Rees_Rat}.

We begin with the following basic convex geometric fact. We include a sketch of its proof as we could no find a reference for it, but first we need the following notation.

\begin{setup}\label{notation_last_sec}
	Let $\Gamma_1\subsetneq \RR^{n_1}_{\gs 0}$ and $\Gamma_2\subsetneq \RR^{n_2}_{\gs 0}$ be two positive polyhedra. We denote by $\conv(\Gamma_1, \Gamma_2)\subseteq \NN^{n_1+n_2}$ the convex hull of $\Gamma_1$ and $\Gamma_2$ after we embed them in $\NN^{n_1+n_2}$  by adding zeroes in the last $n_2$ components and the first $n_1$ components, respectively. For two hyperplanes  $H_1, H_2$ defined by $\langle\bh_1,\fX_1\rangle=h_{1,1}X_{1,1}+\cdots+ h_{1,n_1}X_{1,n_1}=c_1$ and   $\langle\bh_2,\fX_2\rangle=h_{2,1}X_{2,1}+\cdots+ h_{2,n_2}X_{2,n_2}=c_2$, respectively, we denote by $H_1\star H_2$ the hyperplane
	$H_1\star H_2 = \langle c_2\bh_1+c_1\bh_2, (\fX_1,\fX_2)\rangle=c_1c_2$.
\end{setup}

\begin{proposition}\label{subdirect}
Adopt \autoref{notation_last_sec}. $\conv(\Gamma_1, \Gamma_2)$ is a positive polyhedron and its non-coordinate supporting hyperplanes are those of the form $H_1\star H_2$ where $H_1$ and $H_2$ are non-coordinate supporting  hyperplanes of $\Gamma_1$ and $\Gamma_2$, respectively.
\end{proposition}
\begin{proof}
Write $\Gamma_1=\mathscr{Q}_1+\RR_{\gs 0}^{n_1}$ and $\Gamma_2=\mathscr{Q}_2+\RR_{\gs 0}^{n_2}$ where $\mathscr{Q}_1$ and $\mathscr{Q}_2$ are rational  polytopes not containing the origin. By \cite[Proposition 2.3]{McMullen76} the second statement  holds for $\conv(\mathscr{Q}_1, \mathscr{Q}_2)$.
By \cite[Proposition 1.31]{bruns2009polytopes} and its proof we have $\conv(\Gamma_1, \Gamma_2)=\conv(\mathscr{Q}_1, \mathscr{Q}_2)+\RR_{\gs 0}^{n_1+n_2}$, and from this  both statements  follow for $\conv(\Gamma_1, \Gamma_2)$.
\end{proof}

\begin{example}\label{example_Reeses}
Consider the positive polyhedra $\Gamma_1=(4) +\RR_{\gs 0}^{1}$ and
$\Gamma_2=\conv\left((2,1), (3,0)\right)+\RR^2_{\gs 0}$ (cf. \autoref{example_det}). The only non-coordinate  hyperplane $H$ of $\Gamma_1$ has equation $X=4$, and the ones of $\Gamma_2$, denoted by $H_1$ and $H_2$, have equations $X_1+X_2=3$ and $X_1=2$. Therefore, from  \autoref{subdirect} it follows that $\conv(\Gamma_1, \Gamma_2)$ have two non-coordinate  hyperplanes $H\star H_1$ and $H\star H_2$ with equations $4X_1+4X_2+3X = 12$ and $2X_1+X = 4$,  see \autoref{fig3}.
	\begin{figure}[hbt!]
	\centering
	\subfigure{\includegraphics[height=2in]{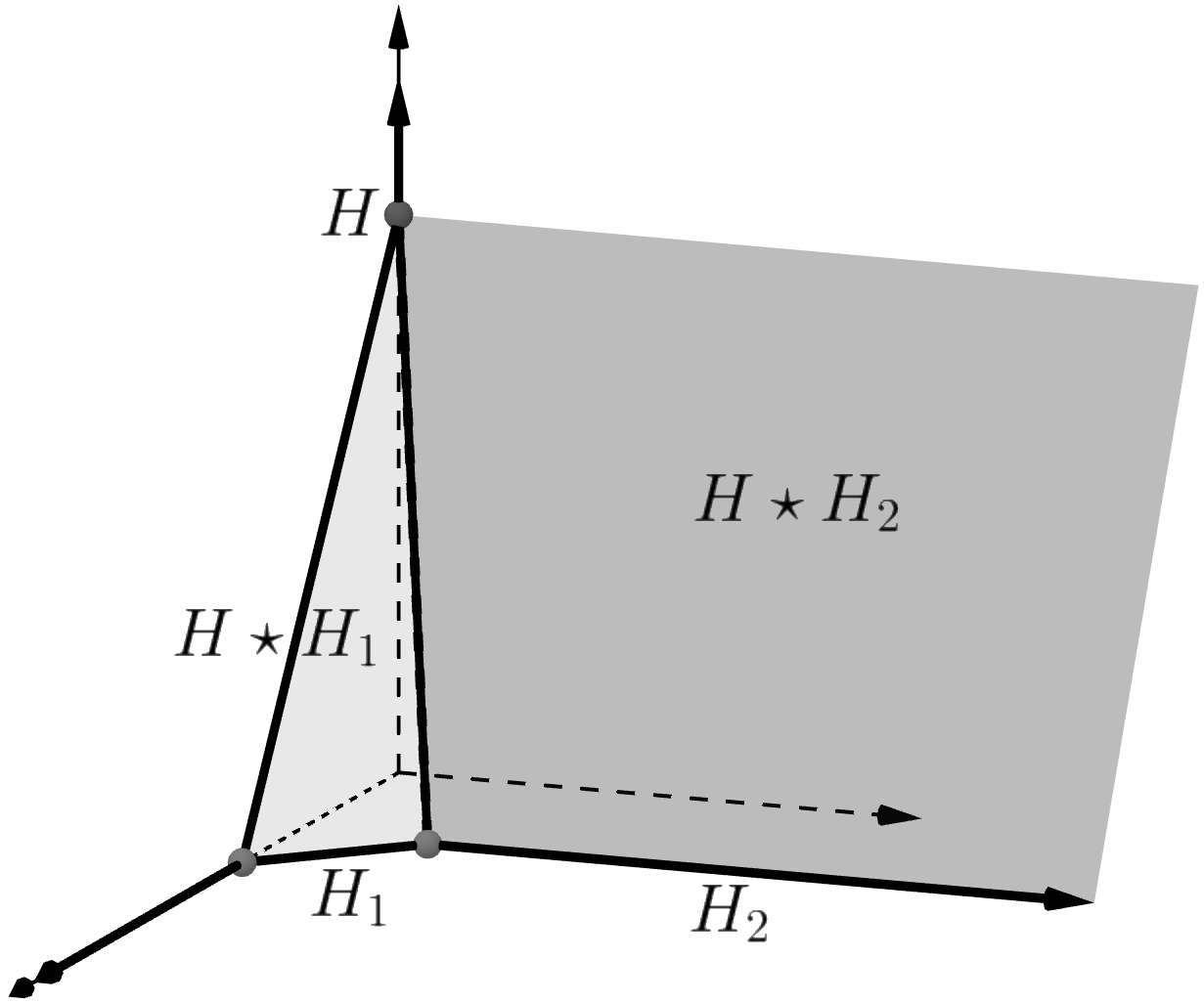}}
	\caption{The convex hull $\conv(\Gamma_1, \Gamma_2)$.} \label{fig3}
\end{figure}
\end{example}

The following setup will be used  in the rest of this section.

\begin{setup}\label{setup_sum}
Let $\kk$ be a field. Let $R$ and $S$  be finitely generated $\kk$-algebras that are domains. Let  $I\subset R$ and $J\subset S$ be nonzero  ideals that have a Rees package (see  \autoref{def_Rees_Package}). Let $(\B,\underline{v}=(v_1,\ldots, v_d), \Gamma)$ and $(\C, \underline{\gamma}=(\gamma_1,\ldots, \gamma_e), \Sigma)$ be Rees packages for $I$ and $J$, respectively. We denote by $\B\otimes \C$ the basis $\{\fb\otimes \mathfrak{c}\mid \fb\in\B, \mathfrak{c}\in \C\}$ of $T:=R\otimes_\kk S$. 
\end{setup}

The following are two technical lemmas that are needed in the main result of this section.

\begin{lemma}\label{sums_vals}
Adopt \autoref{setup_sum}. Let $v$ and $\gamma$ be discrete valuations of $R$ and $S$, respectively, such that $v$ is $\B$-monomial and $\gamma$ is $\C$-monomial. Let $v+ \gamma: T\to \NN$ be the $\B\otimes \C$-monomial function induced by $(v+ \gamma)(\fb\otimes \mathfrak{c})=v(\fb)+\gamma(\mathfrak{c})$. Then $v+ \gamma$ is a valuation.
\end{lemma}
\begin{proof}
Set $V:=v+\gamma$. It suffices to show that for every $f,g\in T$ one has $V(fg)=V(f)+V(g)$. Write
$$f=\sum_{c_{i}\in \kk^*}c_{i}\,\fb_i\otimes \fc_i\,\,\,\quad \text{and} \quad \,\,\,g=\sum_{d_{j}\in \kk^*}d_{j}\,\fb_j\otimes \fc_j.$$
Set $a:=V(f)$ and $b:=V(g)$. We may
assume $V(\fb_i\otimes \fc_i)=a$  and $V(\fb_j\otimes \fc_j)=b$ for every $i,j$.
Assume $V(fg)>a+b$, we will arrive to a contradiction. The latter means that all  the elements of $\B\otimes \C$ of degree $a+b$ appearing in the expansions of
$(\fb_i\otimes \fc_i)(\fb_j\otimes \fc_j)$ cancel in $fg$. Since 
these monomials would still cancel if one passes to the algebraic closure of $\kk$,  we may assume $\kk=\overline{\kk}$.

For each $n\in \NN$ consider the following $R$-ideals
$I_n:=(\{\fb\in \B\mid v(\fb)\gs n\})$ and   $I_{>n}:=(\{\fb\in \B\mid v(\fb)> n\})$, and the $S$-ideals
$J_n:=(\{\fc\in \C\mid \gamma(\fc)\gs n\})$ and  $J_{>n}:=(\{\fc\in \C\mid \gamma(\fc)> n\})$.
Consider the $\kk$-algebras
$G_1 = \oplus_{n\in \NN}I_n/I_{>n}$ and $G_2 = \oplus_{n\in \NN}J_n/J_{>n}$. One easily shows that $G_1$ and $G_2$ are domains. Let $G:=G_1\otimes_\kk G_2$ which is also a domain by \cite[Chapter V, \S17.2, Corollary]{Bourbaki_Algebra_II}. Consider $G$ as an $\NN$-graded $\kk$-algebra with $[G]_m=\oplus_{0\ls n \ls m}[G_1]_n\otimes_{\kk}[G_2]_{m-n}$ for every $m\in \NN$. Thus,  the images  $[f]$, $[g]$ in $G$ are forms of degrees $a$ and $b$, respectively. By assumption one has that $[f][g]=[fg]$ vanishes, which is a contradiction.
\end{proof}

\begin{lemma}\label{vals_are_intc}
Let $R$ be a Noetherian domain and  let $v$ be a discrete valuation of $R$. Let  $m\in \ZZ_{> 0}$ and consider the ideal $K=\{x\in R\mid v(x)\gs m\}$. Then $K$ is integrally closed.
\end{lemma}
\begin{proof}
Assume that there exists $x\in \overline{K}\setminus K$, i.e., $0<v(x)<m$. Let  $x^n+a_1x^{n-1}+\cdots+a_n=0$ with  $a_i \in K^i$ be an integral equation of  $x$.   Since $(n-i)v(x)<(n-i)m$ for every $0\ls i<n$, we have $nv(x)< (n-i)m + iv(x)$. Therefore $v(x^n+a_1x^{n-1}+\cdots+a_n)=nv(x)\in \ZZ_{> 0}$, which is a  contradiction.
\end{proof}

The following is the main ingredient in the proof of the main result of this section.

\begin{proposition}\label{prev_main_package}
Adopt \autoref{setup_sum}. Then $\left(\B\otimes \C, (\underline{v},\underline{\gamma}), \conv(\Gamma,\Sigma)\right)$ is a Rees package for 
$IT +JT$. 
\end{proposition}
\begin{proof}
Set $\Omega:=\conv(\Gamma,\Sigma)$. By \autoref{subdirect} and \autoref{sums_vals} the $\B\otimes \C$-monomial functions $V_k$ in \autoref{prop_Rees_Package} corresponding to the non-coordinate supporting hyperplanes of $\Omega$ are valuations. Thus it suffices  to show that for every $n\in \ZZ_{>0}$ the ideal  $\overline{\left(IT + JT\right)^n}$ is spanned as a $\kk$-vector space by   $\mathcal{W}:=\{\fb\otimes \fc\in \B\otimes \C\mid (\underline{v}(\fb),\underline{\gamma}(\fc)) \in n\Omega\}$. By the above conclusion on the  non-coordinate supporting hyperplanes of $\Omega$ and by \autoref{vals_are_intc} we have that
 $\spann(\W)$ is an integrally closed ideal of $T$. Moreover, $\left(IT + JT\right)^n\subseteq \spann(\W)$ which follows from the fact that 
  for every $w\in \QQ_{>0}$ we have
 \begin{equation}\label{Omega_decomp}
 	w\Omega = \bigcup_{\stackrel{0\ls \alpha\ls w,}{{\scriptscriptstyle \alpha\in \RR}}} \left(\alpha \Gamma + (w-\alpha)\Sigma\right).
 \end{equation}
It remains to show $\spann(\W)\subseteq \overline{\left(IT + JT\right)^n}$.
 Let $\fb\otimes \fc\in \W$ so that $\underline{v}(\fb)\in \alpha \Gamma$ and $\underline{\gamma}(\fc) \in (n-\alpha)\Sigma$ for some  $0\ls \alpha \ls n$, see \autoref{Omega_decomp}. Let $d\in \ZZ$ be such that $\alpha d\in \NN$ and $(n-\alpha)d\in \NN$, then  $\fb^d\in \overline{I^{\alpha d}}$ and
 $\fc^d\in \overline{J^{(n-\alpha) d}}$. Therefore  $(\fb\otimes \fc)^d\in \overline{I^{\alpha d}}\,\overline{J^{(n-\alpha )d}}\,T\subseteq \overline{\left(IT + JT\right)^{nd}}$, where the last inclusion follows from \autoref{base_change}. 
  We conclude that $\fb\otimes \fc\in \overline{\left(IT + JT\right)^n}$, finishing the proof.
\end{proof}

The following is the main result in this section. We note that, by \autoref{discrete}, the sum on the right hand side is finite.

\begin{theorem}\label{main_Rees_P}
Adopt \autoref{setup_sum}. For every $w\in \QQ_{\gs 0}$ we have
\begin{equation}\label{binom_formula}
\overline{\left(IT + JT\right)^w}=\sum_{\stackrel{0\ls \alpha\ls w,}{{\scriptscriptstyle \alpha\in \QQ}}}\overline{I^\alpha}\,\overline{J^{w-\alpha}}\,T = \sum_{\stackrel{0\ls \alpha\ls w,}{{\scriptscriptstyle \alpha\in \QQ}}}\overline{(IT)^\alpha}\,\overline{(JT)^{w-\alpha}}.
\end{equation}
\end{theorem}
\begin{proof}
The first equality follows by combining  \autoref{prev_main_package},   \autoref{prop_Rees_Package}(2), and Equation \autoref{Omega_decomp}.
The second equality follows from \autoref{base_change}. 
\end{proof}

In the following corollary we highlight important classes of ideals for which the previous theorem applies.

\begin{corollary}\label{cor_invs_all}
Adopt \autoref{setup_sum}. If $I$ and $J$ belong to any of the following classes of ideals, and not necessarily both from the same class, then {\rm \autoref{binom_formula}} holds.
\begin{enumerate}[{\rm (a)}]
	\item Monomial ideals in affine semigroup rings.
\item  Sums of products of determinantal ideals of a generic matrix $X_{m\times n}$ and, if $\kk=\CC$,  $\GL_m(\CC) \times \GL_n(\CC)$-invariant ideals of $\CC[X_{m\times n}]$.
\item  Sums of products of determinantal ideals of a  generic symmetric matrix $Y_{n\times n}$ and, if $\kk=\CC$, $\GL_n(\CC)$-invariant ideals of $\CC[Y_{n\times n}]$.
\item Sums of products of  ideals of Pfaffians of a skew-symmetric matrix $Z_{n\times n}$ and, if $\kk=\CC$, $\GL_n(\CC)$-invariant ideals of $\CC[Z_{n\times n}]$.
\item Products of determinantal ideals of a Hankel matrix of variables.
\end{enumerate}
\end{corollary}
\begin{proof}
Combine \autoref{main_Rees_P} with:   \autoref{mon.aff.rees} for (a), \autoref{gen_det_rees}, \autoref{cor_inv_gen} for (b), \autoref{gen_sym_rees}, \autoref{cor_inv_sym} for (c), \autoref{gen_sym_pff}, \autoref{cor_inv_Pf} for (d), and \autoref{gen_Hank} for (e).
\end{proof}

Next we see \autoref{main_Rees_P} and \autoref{cor_invs_all} in action.

\begin{example}\label{example_cool}
Let $T=\kk[x_1^2x_2, x_1x_2^3, y_1, y_2,y_3,y_4,y_5,y_6]$ and $$K= \left(x_1^4x_2^2, x_1^3x_2^4\right) + \left(y_1y_5-y_2y_4, y_1y_6-y_3y_4, y_2y_6-y_3y_5\right) + \left(y_1, y_2,y_3,y_4,y_5,y_6\right)^3\subset T.$$
Notice that $K = IT + JT$ where $I$ is the ideal in \autoref{example_mon_sem} and $J$  the ideal $I_\Lambda$ 
in \autoref{example_det}. We identify these ideals with their images in $T$.
We also set $J_1:=\left(y_1, y_2,y_3,y_4,y_5,y_6\right)$ and $J_2:=\left(y_1y_5-y_2y_4, y_1y_6-y_3y_4, y_2y_6-y_3y_5\right)$.
By \autoref{cor_invs_all} we have
$$\overline{K^{3/2}}
=
\sum_{\stackrel{0\ls \alpha\ls 3/2,}{{\scriptscriptstyle \alpha\in \QQ}}}\overline{I^\alpha}\,\overline{J^{3/2-\alpha}}\,
=
\overline{I^{3/2}} +
\overline{I}\overline{J^{1/2}} +
\overline{I^{1/2}}\overline{J} +
\overline{J^{3/2}},
$$
where the last equality is easily observed  
by inspecting \autoref{fig1}. Finally, using \autoref{example_mon_sem},  \autoref{example_det},  and Macaulay2 \cite{M2} we obtain
\begin{align*}
\overline{K^{3/2}} &=
(x_1^6x_2^3, x_1^5x_2^5) +
(x_1^4x_2^2,x_1^3x_2^4)\overline{J^{1/2}} +
(x_1^2x_2)\overline{J} +
\overline{J^{3/2}}\\
&= (x_1^6x_2^3, x_1^5x_2^5)
+ (x_1^4x_2^2,x_1^3x_2^4)J_1^2
+(x_1^2x_2)\left(J_1^3+J_2\right) +
\left(J_2^{(2)}+ J_1^2J_2+J_1^5\right).
\end{align*}
\end{example}

Now we show that the conclusion of \autoref{main_Rees_P} does not always hold if the ideals $I$ and $J$ belong to the same $\kk$-algebra.

\begin{example} \label{same.ring.no}
	Let $R=\kk[x,y]$ be a polynomial ring over a field $\kk$. Let $I= (xy^3)$ and $J=(x^3y)$ respectively. For all $n\in \ZZ_{> 0}$ we have
	$$x^{4n+2}y^{4n+2}\in \overline{\left(I\ + \,J\right)^{2n+1}} \setminus \sum_{\stackrel{0\ls \alpha\ls 2n+1,}{{\scriptscriptstyle \alpha\in \QQ}}}\overline{I^\alpha}\,\overline{J^{2n+1-\alpha}}.$$
\end{example}

From our methods in this section we also obtain the following interesting result. We recall the definition of the sum of two valuations from \autoref{sums_vals}.

\begin{corollary}\label{cor_Rees_val}
Adopt \autoref{setup_sum}. If  $v\in \RV(I)$ and $\gamma\in \RV(J)$ then  $\gamma(J)v+v(I)\gamma \in \RV(IT + JT)$. Conversely, every Rees valuation of $IT + JT$ can be obtained uniquely this way. In particular, $\left|\RV(IT + JT)\right|=\left|\RV(I)\right|\left|\RV(J)\right|$.
\end{corollary}
\begin{proof}
The result follows from \autoref{prop_Rees_Package}(1), \autoref{prev_main_package}, and \autoref{subdirect}.
\end{proof}

\begin{example}
	Let $I=(x^4)\subset \kk[x]$,  and let $J\subset \kk[y_1, y_2,y_3,y_4,y_5,y_6]$ be as in \autoref{example_cool}.
	By \autoref{example_det}, the polyhedra in the Rees packages of $I$ and $J$ are  $\Gamma_1$ and $\Gamma_2$ as in \autoref{example_Reeses}. Thus by \autoref{cor_Rees_val} (or \autoref{prop_Rees_Package}) and \autoref{example_Reeses}, the Rees valuations of $IT + JT$ are the monomial functions  $3\log_x+4\gamma_1+4\gamma_2$ and  $ \log_x+2\gamma_1$,
	where $\gamma_1$ and $\gamma_2$ are as in \autoref{example_det}, and $\log_x$ extracts the exponent of the variable $x$.
\end{example}

The previous corollary naturally gives rise to the following question.

\begin{question}\label{question_rees_product}
Let $\kk$ be an algebraically closed field.
Let $R$ and $S$  be finitely generated $\kk$-algebras that are domains and $T:=R\otimes_\kk S$. Let   $I\subset R$ and $J\subset S$ be nonzero ideals, do we always have
$\left|\RV(IT + JT)\right|=\left|\RV(I)\right|\left|\RV(J)\right|$?
\end{question}

In the next example we show that the answer to \autoref{question_rees_product} is `no' if we do not assume $\kk=\overline{\kk}$.

\begin{example}\label{example.fields.algebraically.closed}
 Consider the polynomial rings  $R=\QQ[x]$ and $S=\QQ[y]$, and the  prime ideals  $I=(x^2+1)\subset R$,  $J=(y^2+1)\subset S$.
 Since $\overline{R}\left[ \dfrac{x^2+1}{x^2+1} \right]=\overline{R}=R$ and  $IR$ has only one minimal prime we obtain  $\left|\RV(I)\right|=1$. Similarly  $\left|\RV(J)\right|=1$.
 Using  Macaulay2 \cite{M2} we conclude that the ideal $IT + JT= (x^2+1, y^2+1)$ is integrally closed  and that it has  two minimal primes. 
 Therefore $|\R\V(I+J)| \geqslant 2$ by \cite[Discussion 10.1.3]{huneke2006integral}.
\end{example}

\section{ Asymptotic rational powers over algebraically closed fields}\label{sec_asymp_form}

This section is devoted to the proof of \autoref{thm_asymp_main} (or \autoref{thmA}).
We will use the following notation throughout.

\begin{setup}\label{notation_main}
	Let $\kk$ be an algebraically closed field. Let $R$ and $S$  be 
	$\kk$-algebras, and  $I\subset R$ and $J\subset S$  nonzero ideals. We denote by $IT + JT$ the $T$-ideal $IT+JT$, where
	$T:=R\otimes_\kk S$.
\end{setup}

%

The following is a main theorem of this section.  
We note that, by \autoref{discrete}, the sum on the left hand side is finite.

\begin{theorem}\label{thm_asymp_main}
	Adopt \autoref{notation_main} with $R$ and $S$ being  finitely generated $\kk$-algebras and normal  domains. 
	There exists   $w_0\in \QQ_{> 0}$ such that if $w_0\ls \tau \ls w-w_0$, $\tau\in \QQ$ then 
	$$\sum_{\stackrel{0\ls \alpha\ls w,}{{\scriptscriptstyle \alpha\in \QQ}}}\overline{I^\alpha}\,\overline{J^{w-\alpha}}\,T
	\subseteq
	\overline{\left(IT + JT \right)^w}
	\subseteq
	\overline{I^\tau}T + \overline{J^{w-\tau}}T
	.$$
\end{theorem}

\begin{remark}\label{remark_weaker}
We note that by
\cite[proof of Lemma 1.2]{mustacta2002multiplier} 
we have
$$\sum_{\stackrel{0\ls \tau\ls w,}{{\scriptscriptstyle \tau\in \QQ}}}\overline{I^\tau}\,\overline{J^{w-\tau}}\,T
= \bigcap_{\stackrel{0\ls \tau\ls w,}{{\scriptscriptstyle \tau\in \QQ}}}\left(	\overline{I^\tau}T + \overline{J^{w-\tau}}T\right),
$$
therefore \autoref{thm_asymp_main} can be seen as a weaker version of the summation formula in \autoref{eqn_The_Q}.
\end{remark}

Before presenting the proof of the theorem, we need some preparatory results. We start with the following statement, which is a
%
%
generalization of \cite[Theorem 2.3]{hubl2008adjoints} that  is suitable for our purposes. First, we need the following definitions.

\begin{definition}[{\cite{hubl2008adjoints}}]\label{gsop_NP}
	Let $R$ be a regular domain and $\fx=x_1,\ldots, x_d$ a permutable regular sequence   of $R$.  We say $\fx$ is a {\it generalized regular system of parameters (grsop)}, if  for every subset $\fx'=x_{i_1}, \ldots, x_{i_r}$ of $\fx$ the ring  $R/(\fx')$ is also a regular domain. Let $\fx$ be an grsop and let $I$ be an $R$-ideal generated by monomials in $\fx$, i.e., $I=\left(\fx^{\ba_1},\ldots, \fx^{\ba_s}\right)$, where if $\ba=(a_1,\ldots, a_d)\in \NN^d$ then $\fx^\ba:=x_1^{a_1}\cdots x_d^{a_d}$. The {\it Newton polyhedron of $I$} is defined as
	$$\NP(I):=\conv\left(\ba_1,\ldots, \ba_s\right)+\RR^d_{\gs 0}\subseteq \RR^d.$$
\end{definition}

\begin{theorem}\label{thm_gen_regular}
	Let $R$, $\fx$, and  $I$ be as in \autoref{gsop_NP}. For every $w\in \QQ_{\gs 0}$ the ideal $\overline{I^w}$ is  generated by the monomials $\fx^\fe$ with  $\fe\in w\NP(I)$.
\end{theorem}
\begin{proof}
	We only  include relevant details of how a modification of the proof of \cite[Theorem 2.3]{hubl2008adjoints}  leads to our statement.  We refer the reader to the proof of \cite[Theorem 2.3]{hubl2008adjoints} for any undefined notation, which we will follow closely.
	
	Let $\fe\in w\NP(I)$ and $c_1,\ldots,c_s\in \QQ_{\gs 0}$ be such that $\sum_i c_i=w$ and $\fe\gs \sum_i c_i\ba_i$ componentwise. Write $c_i=\frac{m_i}{n}$ and $w=\frac{p}{q}$ where $m_i,p\in \NN$ and $n,q\in \ZZ_{>0}$.
	Therefore, $qn\fe\gs \sum_i qm_i\ba_i$. It follows that
	$$\left( \fx^{q\fe}\right)^n\in I^{\sum_i qm_i}=\left(I^p\right)^n,$$
	and then
	$\left(\fx^\fe\right)^q\in \overline{I^p}$, i.e.,
	$\fx^\fe \in \overline{I^w}$.
	
	We proceed with the reverse inclusion. Let $H$ be a non-coordinate supporting hyperplane of $\NP(I)$ defined by $\langle\bh,\fX\rangle=h$ where $\bh=(h_1,\ldots, h_d)\in \NN^d$, $h\in \ZZ_{>0}$, and $\fX=(X_1,\ldots, X_d)$. Let $I_H$ be the ideal generated by the monomials $\fx^\fe$ with $\langle \bh, \fe\rangle \gs h$. It suffices to show that the conclusion of the theorem holds for $I_H$. Let $\frp=(Y_1,\ldots, Y_t)R'$. By \autoref{assoc_prime} and \cite[Corollary 5.4.2]{huneke2006integral}, the rational powers of $\frp^hR'_\frq$ have no embedded primes for every prime ideal $\frp\subseteq \frq$. Thus, it suffices to show that
	$\overline{\left(\frp^hR'_\frp\right)^w}=\frp^{\lceil wh\rceil}R'_\frp$, which follows from $\frp R'_\frp$ being a prime ideal generated by a regular sequence, as in this case the powers $\frp^n R'_\frp$ are integrally closed and  $\gr_{R'_\frp}(\frp R'_\frp):=\oplus_{n\in \NN} \frp^{n}R'_\frp/\frp^{n+1}R'_\frp$ is a domain.
\end{proof}

\begin{corollary}\label{dvr}
	Let $\kk$ be an algebraically closed field. 
	Let $(V_1,\eta_1)$ and $(V_2,\eta_2)$ be two divisorial valuation rings containing $\kk$ and such that  $V:=V_1\otimes_\kk V_2$ is Noetherian.  Let $I=(\eta_1^{n_1})$ and $J=(\eta_2^{n_2})$ be ideals of $V_1$ and $V_2$, respectively. For every $\tau,w\in \QQ_{\gs 0}$ with $\tau\ls w$, we have the following inclusion of $V$-ideals
	$$\overline{\left(IT + JT\right)^w}\subseteq \overline{I^\tau}T + \overline{J^{w-\tau}}T.$$
\end{corollary}
\begin{proof}
	By \cite[Theorem 6(e)]{tousi2003tensor} and \cite[Chapter V, \S17.2, Corollary]{Bourbaki_Algebra_II}, $V$ is a regular domain and $\eta_1,\eta_2$ is a grsop of $V$ (see \autoref{gsop_NP}). Thus by \autoref{thm_gen_regular} the ideal $\overline{\left(IT + JT\right)^w}$ is generated by the monomials $\eta_1^{e_1}\eta_2^{e_2}$ with $n_2e_1+n_1e_2\gs wn_1n_2$.  Therefore, for any such $\eta_1^{e_1}\eta_2^{e_2}$ we must have $e_1\gs \tau n_1$ or $e_2\gs (w-\tau) n_2$. The result follows  as $\overline{I^\tau}$ is generated by $\eta_1^{\lceil\tau n_1\rceil}$ and  $\overline{J^{w-\tau}} $ by $\eta_2^{\lceil(w-\tau) n_2
		\rceil}$.
\end{proof}

The following lemma offers a vanishing result for rational powers of exceptional divisors of normalized blowups that can be of independent utility and interest.

\begin{lemma}\label{vanishing}
Adopt \autoref{setup_setup}. Let $v: X^+\longrightarrow X$ be the normalization of the blowup of $X$ along $\I$, and let $E$ be its exceptional divisor, i.e., $\I\cdot \O_{X^+}=\O_{X^+}(-E)$. 
Then  there exists a positive integer $n$ such that if $\beta\in \QQ$ and $\beta\gs n$ we have $R^1v_{*}\left(\O_{X^+}(-\lceil\beta E\rceil)\right)=0$.
\end{lemma}
\begin{proof}
Set $\cT:=\O_{X^+}(-E)$. By \autoref{alternative}	we have $\overline{\cT^\beta}=\O_{X^+}\left(-\lceil\beta E\rceil\right)$ for every $\beta\in \QQ_{\gs 0}$. Notice that there exists $m\in \ZZ_{>0}$ such that for every $\beta \gs m$ we have
$\overline{\left(\cT^+\right)^{\beta+m}}
=
\overline{\left(\cT^+\right)^{\beta}}\,\,\overline{\left(\cT^+\right)^{m}}$.
Indeed, for any finite affine cover $\{U_\lambda\}_{\lambda\in\Lambda}$ of $X^+$, this holds locally on each $U_\lambda$ by \autoref{Noeth_rat_power}(3), producing finitely many $m_\lambda$. So, we take can take $m:=\lcm(m_\lambda\mid \lambda\in\Lambda)$.  Therefore for every $\beta\gs m$ we have
\begin{equation}\label{eq1}
	\overline{\left(\cT^+\right)^{\beta}}
	=
	\overline{\left(\cT^+\right)^{\beta-m
			\lfloor\frac{\beta}{m}\rfloor+m}}\,\,
	\left(\overline{\left(\cT^+\right)^{m}}\right)^{\lfloor\frac{\beta}{m}\rfloor-1}
	\cong
	\overline{\left(\cT^+\right)^{\beta-m\lfloor \beta/m\rfloor+m}}\otimes \O_{X^+}(-E)^{\otimes m\lfloor \beta/m\rfloor-m},
\end{equation}
where the last isomorphism holds by flatness.
%
Since $m\ls \beta-m\lfloor \beta/m\rfloor+m\ls 2m$, by \autoref{discrete} there are finitely many distinct ideal sheaves of the form $	\overline{\left(\cT^+\right)^{\beta-m\lfloor \beta/m\rfloor+m}}$.
Since $\O_{X^+}(-E)$ is  ample,  
by Serre's vanishing theorem \cite[Proposition III.5.3]{hartshorne1966}  there exists $n\in \ZZ_{>0}$ such that for every $\beta \gs n$
\begin{equation}\label{eq2}
	R^1v_{*}\left(
	\overline{\left(\cT^+\right)^{\beta-m\lfloor \beta/m\rfloor+m}}\otimes \O_{X^+}(-E)^{\otimes m\lfloor \beta/m\rfloor-m}
	\right) =0.
\end{equation}
The result now follows by combining \autoref{eq1} and \autoref{eq2}. 
\end{proof}

We are now ready to prove the main theorem of this section.

\begin{proof}[Proof of \autoref{thm_asymp_main}]\label{proof_thm_asymp}
	
	
	
	Fix $\alpha, w\in \QQ_{\gs 0}$ with  $0\ls \alpha\ls w$.  By \autoref{base_change} 
	we have
	$$\overline{I^\alpha}\,\overline{J^{w-\alpha}}T= \overline{(IT)^\alpha}\,\,\overline{(JT)^{w-\alpha}}
	\subseteq
	\overline{(IT + JT)^w},$$
	proving the first inclusion.
	 We proceed with the second inclusion to show that
	 there exists $w_0\in \QQ_{>0}$ such that
	\begin{equation}\label{what_to_show}
		\overline{\left(IT + JT\right)^w}\subseteq \overline{I^\tau}T+\overline{J^{w-\tau}}T
	\end{equation}
	for all $w_0\ls \tau\ls w-w_0$, $\tau\in \QQ$. 
	Let $X_1=\Spec(R)$ and $X_2=\Spec(S)$ so that $X_1\times X_2=\Spec(T)$. Let $\I=\tilde{I}\subseteq O_{X_1}$ and $\mathcal{T}=\tilde{J}\subseteq O_{X_2}$ be the sheaves associated to  $I$ and $J$, respectively. Let $v_1:X_1^+\longrightarrow X_1$ and $v_2:X_2^+\longrightarrow X_2$ be the normalization of the blowups of $X_1$ and $X_2$ along $\I$ and $\T$, respectively, so that
	$\I^+:=\I\cdot \O_{X_1^+}=:\O_{X_1^+}(-E_1)$ and  $\T^+:=\T\cdot \O_{X_2^+}=:\O_{X_2^+}(-E_2)$.
	Moreover, consider $Y:=X_1\times X_2$, $Y^+:=X_1^+\times X_2^+$, the canonical projections $$p_1:Y\longrightarrow X_1, \,\,\,\,p_2:Y\longrightarrow X_2, \,\, \,\,q_1:Y^+\longrightarrow X_1^+,\,\, \,\, q_2:Y^+\longrightarrow X_2^+,$$ and $f:=v_1\times v_2:Y^+ \longrightarrow Y$. These maps are represented in the following diagram.
	\begin{equation*}
		\begin{tikzcd}
			X_1^+  \arrow[d, "v_1"'] & \arrow[l, "q_1"']  Y^+=X_1^+\times X_2^+ \arrow{r}{q_2} \arrow{d}{f=v_1\times v_2} & X_2^+ \arrow{d}{v_2} \\
			X_1 &  \arrow{l}{p_1} Y=X_1\times X_2 \arrow[r, "p_2"'] & X_2
		\end{tikzcd}
	\end{equation*}

	
	
	We note that it suffices to show that the following two claims are true. Here, for simplicity of notation, for an ideal sheaf $\mathcal{J}\subseteq O_{X_i^+}$, resp.  $\mathcal{J}\subseteq O_{X_i}$, we write $q_i^{-1}\mathcal{J}$, resp. $p^{-1}_i\mathcal{J}$, for the ideal sheaf $q_i^{-1}\mathcal{J}\cdot \O_{Y^+}$, resp. $p_i^{-1}\mathcal{J}\cdot \O_{Y}$. We note that, since each $q_i$, $p_i$ is flat, the latter are isomorphic to $q_i^{*}\mathcal{J}$, resp.  $p_i^{*}\mathcal{J}$.
	
	\noindent{\bf Claim 1:} For any rational number $0\ls \tau \ls w$ we  have
	$$\overline{\left( q_1^{-1}\I^++  q_2^{-1}\T^+ \right)^w}\subseteq
	q_1^{-1}\overline{\left(\I^+\right)^\tau}
	+
	q_2^{-1}\overline{\left(\T^+\right)^{w-\tau}}.$$
	{\bf Claim 2:} There exists $w_0\in \QQ_{>0}$ such that if $w_0\ls \tau\ls w-w_0$, $\tau\in \QQ$ then
	$$f_*\left(q_1^{-1}\overline{\left(\I^+\right)^\tau}
	+
	q_2^{-1}\overline{\left(\T^+\right)^{w-\tau}}\right)=
	f_*\left(q_1^{-1}\overline{\left(\I^+\right)^\tau}\right)
	+
	f_*\left( q_2^{-1}\overline{\left(\T^+\right)^{w-\tau}}\right).
	$$
	Indeed, if both claims hold,  since $ q_1^{-1}\I^++  q_2^{-1}\T^+ =f^{-1}\left( p_1^{-1}\I+  p_2^{-1}\T\right)\cdot \O_{Y^+}$, then from \autoref{birr_transf} one obtains
	\begin{align*}
		\overline{( p_1^{-1}\I+  p_2^{-1}\T)^w}=f_*\left(\overline{\left( q_1^{-1}\I^++  q_2^{-1}\T^+ \right)^w}\right)&\subseteq
		f_*\left(q_1^{-1}\overline{\left(\I^+\right)^\tau}\right)
		+
		f_*\left( q_2^{-1}\overline{\left(\T^+\right)^{w-\tau}}\right) \\
		&= p_1^{-1}\overline{\left(\I\right)^\tau}
		+
		p_2^{-1}\overline{\left(\T\right)^{w-\tau}},
	\end{align*}
     for $w_0\ls \tau\ls w-w_0$, $\tau\in \QQ$.
	The last equality follows from K\"unneth formula (see \cite[tag 0BEC, Lemma 33.29.1]{stacks-project} and \cite[Proposition III.8.5]{Har}) using \autoref{birr_transf} and the facts that $v_{1*}\left(\O_{X_1^+}\right)=\O_{X_1}$ and $v_{2*}\left(\O_{X_2^+}\right)=\O_{X_2}$ \cite[proof of Corollary III.11.4]{Har}. Thus, \autoref{what_to_show} will follow by taking global sections, see \autoref{global}.
	
	We proceed to prove Claim 1, for which we can restrict to a product of affine schemes $U_1\times U_2\subset Y^+$, where $U_1=\Spec(A_1)\subset X_1^+$, $V=\Spec(A_2)\subset X_2^+$, and $A_1$ and $A_2$ are Noetherian normal domains. Therefore $\I^+$ and $\T^+$ can be identified with principal 
	ideals $(a_1)\subseteq A_1$ and $(a_2)\subseteq A_2$, respectively, and $q_1^{-1}\I^++  q_2^{-1}\T^+$ with $(a_1)A + (a_2)A$ where $A:= A_1\otimes_\kk A_2$.  By \autoref{assoc_prime} and \cite[Corollary 5.4.2]{huneke2006integral} the associated primes of the rational powers $\overline{(a_1)^\tau}$ and $\overline{(a_2)^{w-\tau}}$ all have codimension one. By \cite[Corollary 2.8]{ha2020symbolic} the associated primes of $(a_1)A + (a_2)A$ are sums $\frp_1A + \frp_2A$  with $\frp_i$ associated prime of $(a_i)$, $i=1,2$. Moreover,  rational powers localize  by \autoref{localize}.  Thus,  to show $\overline{\left((a_1)A + (a_2)A\right)^w}\subseteq
	\overline{(a_1)^\tau}A + \overline{(a_2)^{w-\tau}}A$, it suffices to do so in the localizations $(A_1)_{\frp_1}\otimes_\kk (A_2)_{\frp_2}$. But the latter  follows from \autoref{dvr}, finishing the proof of Claim 1.
	
	We continue with the proof of Claim 2. Consider the exact sequence
	\begin{align*}
		0\longrightarrow q_1^{-1}\overline{\left(\I^+\right)^\tau}
		\cap
		q_2^{-1}\overline{\left(\T^+\right)^{w-\tau}}
		&\longrightarrow
		q_1^{-1}\overline{\left(\I^+\right)^\tau}
		\oplus
		q_2^{-1}\overline{\left(\T^+\right)^{w-\tau}}\\
		&\longrightarrow
		q_1^{-1}\overline{\left(\I^+\right)^\tau}
		+
		q_2^{-1}\overline{\left(\T^+\right)^{w-\tau}}
		\longrightarrow 0.
	\end{align*}
	Applying $f_*$, the claim reduces to showing
	$$R^1f_*\left(q_1^{-1}\overline{\left(\I^+\right)^\tau}
	\cap
	q_2^{-1}\overline{\left(\T^+\right)^{w-\tau}}\right)= R^1f_*\left(q_1^{-1}\overline{\left(\I^+\right)^\tau}
	\otimes
	q_2^{-1}\overline{\left(\T^+\right)^{w-\tau}}\right)=0.$$
	By K\"unneth formula it suffices to show that there exists $w_0\in \QQ_{> 0}$ such that if $\beta\in \QQ$ and $\beta\gs w_0$ then
	\begin{equation}\label{eq:vanishing}
	R^1v_{1*}\left(\overline{\left(\I^+\right)^\beta}\right)=0 \qquad \text{and} \qquad
R^1v_{2*}\left(\overline{\left(\T^+\right)^\beta}\right)=0.
	\end{equation}
The latter vanishings hold by \autoref{vanishing}, concluding the proof of Claim 2.  This finishes the proof of the theorem.
\end{proof}

\begin{remark}
We note that if for certain ideals $I$ and $J$ the vanishings in \eqref{eq:vanishing} hold for every $\beta\gs 0$, then the proof of \autoref{thm_asymp_main} shows that these ideals satisfy the conclusion of \autoref{The_Question}. 
\end{remark}

By working locally in affine open subsets, \autoref{thm_asymp_main} implies the following more general statement (cf. \cite[Theorem 0.3]{mustacta2002multiplier}). As in the proof of  \autoref{thm_asymp_main}, we denote by $p^{-1}\mathcal{J}$, resp. $q^{-1}\mathcal{J}$, the ideal sheaves $p^{-1}\mathcal{J}\cdot \O_{X\times Y}$, resp. $q^{-1}\mathcal{J}\cdot \O_{X\times Y}$.

\begin{corollary}\label{cor_binomial_main}
	Let $X, Y$ be 
	normal varieties over an algebraically closed field $\kk$. 
	Let $\I\subseteq \O_X,  \T\subseteq \O_y$ be nonzero ideal sheaves,  and $p: X\times Y\longrightarrow X$, $q: X\times Y\longrightarrow Y$ the canonical projections.	There exists   $w_0\in \QQ_{> 0}$ such that if $w_0\ls \tau \ls w-w_0$, $\tau\in \QQ$ then
	$$
		\sum_{\stackrel{0\ls \alpha\ls w,}{{\scriptscriptstyle \alpha\in \QQ}}}
	p^{-1}\overline{\I^\alpha}\,\,
	q^{-1}\overline{\T^{w-\alpha}}
	\subseteq
	\overline{\left(p^{-1}\I\,+\,q^{-1}\T\right)^w}
	\subseteq
	p^{-1}\overline{\left(\I\right)^\tau}\,+\,
	q^{-1}\overline{\left(\T\right)^{w-\tau}}.
	$$
\end{corollary}

Similarly as in \cite[Proposition 2.1]{mustacta2002multiplier}, the formula in \autoref{thm_asymp_main} can be used to approximate rational powers up to $\fm$-adic neighborhoods. This is also related to the main theorem in \cite{delSwan97, delSwan97err}.

\begin{corollary}
	Adopt \autoref{notation_main} with $R$ a  finitely generated $\kk$-algebra and normal  domain, and $S$ a polynomial ring with homogeneous maximal ideal $\fm$. For every  $n\in \NN$, there exists  $w_0=w_0(n)\in \QQ_{> 0}$ such that if $w,\ell\gs w_0$ then
	$$\overline{\left(IT+\fm^nT\right)^{w+\ell}}\subseteq
	\overline{I^w}T+\fm^{\lceil n\ell\rceil}T.$$
	Moreover, for any  nonzero ideal $J\subset S$
	 there exists a  $w_0\in \QQ_{> 0}$ such that if $w,\ell\gs w_0$ then
	$$\overline{\left(IT + JT\right)^{w+\ell}}\subseteq
	\overline{I^w}T+J^{\lfloor \ell\rfloor -\dim(S)+1}T.$$
\end{corollary}
\begin{proof}
	The first statement follows from \autoref{thm_asymp_main} and 
	since the powers of $\fm$ are integrally closed, then $\overline{\left(\fm^n\right)^\ell}=\fm^{\lceil n\ell\rceil}$. The second statement also follows from  \autoref{thm_asymp_main} and by the inclusions $\overline{J^{\ell}}\subseteq \overline{J^{\lfloor \ell\rfloor }}\subseteq J^{\lfloor \ell\rfloor -\dim(S)+1}$, which hold by the {B}rian\c{c}on-{S}koda Theorem \cite{LipmSatha81} (see also \cite[Corollary 13.3.4]{huneke2006integral}).
\end{proof}



\bibliographystyle{plain}

\bibliography{References}


\end{document}